\documentclass[a4paper, reqno]{amsart}
\usepackage{wrapfig}
\usepackage[dvips]{graphicx}
\usepackage{amsmath,amsthm,amssymb,amscd}
\usepackage{mathrsfs}
\usepackage[all]{xy}
\theoremstyle{definition}
\newtheorem{thm}{Theorem}[section]
\newtheorem{Def}[thm]{Definition}
\newtheorem{pro}[thm]{Proposition}
\newtheorem{cor}[thm]{Corollary}
\newtheorem{lem}[thm]{Lemma}

\newtheorem{rem}[thm]{Remark}

\newtheorem*{mainthm}{Theorem A}
\newtheorem*{mainthm2}{Theorem B}
\newtheorem*{mainthm3}{Theorem C}
\theoremstyle{definition}

\begin{document}
\title{Aperiodic automorphisms of the Razak--Jacelon algebra} 
\author{Norio Nawata}
\address{Department of Pure and Applied Mathematics, Graduate School of Information Science 
and Technology, The University of Osaka, Yamadaoka 1-5, Suita, Osaka 565-0871, Japan}
\email{nawata@ist.osaka-u.ac.jp}
\thanks{This work was supported by JSPS KAKENHI Grant Number JP26K06829}

\begin{abstract}
We completely classify aperiodic automorphisms of the Razak--Jacelon algebra $\mathcal{W}$ 
up to outer conjugacy. As a technical ingredient, based on the works of Connes, 
Herman--Ocneanu, Kishimoto, and Sato, we introduce a notion of weak stability for 
automorphisms of simple separable monotracial C$^*$-algebras. 
We also consider automorphisms of UHF algebras of 
infinite type and the Jiang--Su algebra. In particular, we classify aperiodic 
``UHF-absorbing'' automorphisms of UHF algebras of infinite type 
up to outer conjugacy. 
\end{abstract}
\maketitle
\section{Introduction}

The Razak--Jacelon algebra $\mathcal{W}$ \cite{J, Raz} is the unique $\mathcal{Z}$-stable simple 
separable nuclear monotracial C$^*$-algebra that is $KK$-equivalent to $\{0\}$ 
(see \cite{CE, EGLN, Na4}). In particular, 
if $A$ is a simple separable nuclear monotracial C$^*$-algebra, then $A\otimes \mathcal{W}$ is 
isomorphic to $\mathcal{W}$. 
The C$^*$-algebra $\mathcal{W}$ can be regarded as a C$^*$-analog of the injective II$_1$ factor and 
a monotracial analog of the Cuntz algebra $\mathcal{O}_2$. In this paper, we completely classify 
aperiodic automorphisms of the Razak--Jacelon algebra $\mathcal{W}$ up to outer conjugacy. 
To the best of the author's knowledge, this is the first complete classification of aperiodic 
automorphisms of a simple separable non-type I stably finite C$^*$-algebra.

The study of automorphisms (more generally, automorphism groups)  
is one of the central themes in the theory of operator algebras. 
Connes \cite{C2, C3} classified automorphisms $\delta$ of the injective  II$_1$ factor up to 
outer conjugacy by the invariants $N(\delta)$ and $\Lambda(\delta)$ 
(see Section \ref{subsec:invariant} for the precise definitions).  
Moreover, he classified automorphisms of the injective II$_{\infty}$ factor by using this classification. 
Kishimoto \cite{Kis0} classified ``strongly aperiodic'' automorphisms of uniformly hyperfinite
(UHF) algebras. 
Furthermore, Evans and Kishimoto \cite{EK} classified trace scaling automorphisms of 
stable approximately finite-dimensional (AF) algebras with totally ordered $K_0$-groups. 
Based on the strategy in \cite{Kis0, EK}, 
``strongly aperiodic'' automorphisms and trace scaling automorphisms of larger classes of 
stably finite C$^*$-algebras have been classified (see, for example, \cite{BK, Kis1, Kis2, M1, Na2, Sa, 
Sza6}). On the other hand, Nakamura \cite{Nak} completely classified aperiodic automorphisms 
of Kirchberg algebras, also based on the strategy in \cite{Kis0, EK}. 
More generally, Gabe and Szab\'o \cite{GS} classified amenable outer actions of discrete groups 
on Kirchberg algebras. (Note that they also consider actions of more general locally compact groups.)

Note that ``strongly aperiodic'' automorphisms and trace scaling automorphisms 
induce aperiodic automorphisms $\delta$ (that is,  $N(\delta)=\infty$) of the von Neumann 
algebras associated with tracial representations, 
and the Rohlin property plays a central role in their classification. 
These classification results are remarkable in themselves. 
However, there exist many interesting automorphisms of UHF algebras that are not strongly 
aperiodic. Indeed, every monotracial $\mathcal{Z}$-stable C$^*$-algebra has infinitely many aperiodic
automorphisms that are not strongly aperiodic and are mutually not outer conjugate 
(see Lemma~\ref{lem:weakly-outer}). 
Our classification covers all aperiodic automorphisms of $\mathcal{W}$,  
including those that are not strongly aperiodic. 

Our classification invariants of aperiodic automorphisms $\alpha$ of $\mathcal{W}$ are 
Connes' invariants $N(\tilde{\alpha})$ and $\Lambda(\tilde{\alpha})$, where 
$\tilde{\alpha}$ is the automorphism of $\pi_{\tau_{\mathcal{W}}}(\mathcal{W})^{''}$ 
induced by $\alpha$. More precisely, we obtain the following theorem. 

\begin{mainthm}[Theorem \ref{thm:main-classification}]
(1) Let $\alpha$ and $\beta$ be aperiodic automorphisms of $\mathcal{W}$. Then 
$\alpha$ and $\beta$ are outer conjugate if and only if 
$N(\tilde{\alpha})=N(\tilde{\beta})$ and $\Lambda(\tilde{\alpha})=\Lambda(\tilde{\beta})$. \ \\
(2) For any $N\in\mathbb{N}$ and $\lambda\in\mathbb{T}$ satisfying $\lambda^{N}=1$, there 
exists an aperiodic automorphism $\gamma$ of $\mathcal{W}$ such that 
$N(\tilde{\gamma})=N$ and $\Lambda (\tilde{\gamma})=\lambda$. 
Furthermore, there exists an aperiodic automorphism $\gamma$ of $\mathcal{W}$ 
such that $N(\tilde{\gamma})=\infty$. 
\end{mainthm}

This classification can be regarded as a rigidity result. Indeed, we obtain the following theorem 
as a corollary of Theorem A.

\begin{mainthm2}[Corollary \ref{cor:rigidity}] 
Let $\alpha$ and $\beta$ be aperiodic automorphisms of $\mathcal{W}$. Then 
$\alpha$ and $\beta$ are outer conjugate if and only if  
$\tilde{\alpha}$ and $\tilde{\beta}$ are outer conjugate.
\end{mainthm2}

Beyond the case of $\mathcal{W}$, 
we also classify aperiodic UHF-absorbing automorphisms of UHF algebras of infinite type. 
More precisely, we obtain the following theorem. 

\begin{mainthm3}[Theorem \ref{thm:classification-uhf}]
Let $B$ be a UHF algebra of infinite type, and let $\alpha$ and $\beta$ be aperiodic automorphisms 
of $B$. Then $\alpha\otimes \mathrm{id}_B$ and $\beta\otimes 
\mathrm{id}_B$ are outer conjugate if and only if 
$N(\tilde{\alpha})=N(\tilde{\beta})$ and $\Lambda(\tilde{\alpha})=\Lambda(\tilde{\beta})$.
\end{mainthm3}

Note that we do not use Kishimoto's classification theorem \cite{Kis0} in the proof of 
Theorem C. In particular, our proof does not use the C$^*$-algebraic 
Rohlin property even for strongly aperiodic automorphisms.  

As a technical ingredient of the classification theorem, we introduce a notion of weak stability for 
automorphisms of simple separable monotracial C$^*$-algebras. 
This property is a weaker version of stability and is inspired by 
the works of Connes, Herman--Ocneanu, Kishimoto, and Sato. 
Roughly speaking, weak stability enables us to reduce cohomology vanishing problems 
in the central sequence C$^*$-algebra to those in the central sequence algebra of the 
injective II$_1$ factor. 

The most difficult part of this paper is to show that every aperiodic automorphism of $\mathcal{W}$ 
is weakly stable. 
Using Sato's strategy \cite{Sa} (see also \cite{MS1, W}), 
we can reduce this problem to the case 
of UHF-absorbing automorphisms. By considering small Rohlin towers \cite{K, Sa2, Sza6} and 
properties of the central sequence C$^*$-algebra of $\mathcal{W}$ \cite{Na2, Na5}, 
we can show that  every UHF-absorbing automorphism of $\mathcal{W}$ is weakly stable.

This paper is organized as follows. In Section \ref{sec:pre}, we collect notation, basic definitions, and 
previous results. In Section \ref{sec:weak}, we introduce a notion of weak stability for 
automorphisms of simple separable monotracial C$^*$-algebras 
and prove absorption results for weakly stable automorphisms. 
In Section \ref{sec:uhf}, we show that every aperiodic UHF-absorbing automorphism of $\mathcal{W}$ 
is weakly stable. Furthermore, we show that every aperiodic UHF-absorbing automorphism 
of $\mathcal{W}$ is $\mathcal{W}$-absorbing. 
In Section \ref{sec:classification}, we show that every aperiodic automorphism of $\mathcal{W}$ is 
weakly stable and prove the main classification theorem and Theorem C. 

\section{Preliminaries}\label{sec:pre}

\subsection{Automorphisms and monotracial C$^*$-algebras}

Throughout this paper, all C$^*$-algebras denoted by $A$ or $B$ 
are assumed to be infinite dimensional unless otherwise specified. Let $A_{+}$ be the set of positive 
elements in $A$. 
We denote by $A^{\sim}$ the unitization of $A$ and by 
$M(A)$ the multiplier algebra of $A$. Note that if $A$ is unital, then $A=A^{\sim}=M(A)$. 
Every automorphism $\alpha$ of $A$ extends uniquely to automorphisms of $A^{\sim}$ and $M(A)$. 
More generally, every nondegenerate homomorphism $\varphi$ from $A$ to $B$ 
extends uniquely to unital homomorphisms from $A^{\sim}$ to $B^{\sim}$ and 
from $M(A)$ to $M(B)$.  We denote these extensions by the same symbols 
$\alpha$ and $\varphi$ for simplicity. 
Let $\mathcal{U}(M(A))$ and $\mathrm{Aut}(A)$ denote the group of unitary elements in $M(A)$ and 
the group of automorphisms of $A$, respectively. 
For $u\in \mathcal{U}(M(A))$, 
define an automorphism $\mathrm{Ad}(u)$ of $A$ by $\mathrm{Ad}(u)(x)=uxu^*$ 
for any $x\in A$, and put $\mathrm{Inn}(A):=\{\mathrm{Ad}(u)\in \mathrm{Aut}(A)\; |\; u\in 
\mathcal{U}(M(A))\}$. 
We say that an automorphism $\alpha$ is \textit{aperiodic} if $\alpha^{n}\notin \mathrm{Inn}(A)$ for 
any $n\in\mathbb{Z}\setminus\{0\}$. An automorphism $\alpha$ is said to be \textit{approximately 
inner} if $\alpha\in \overline{\mathrm{Inn}(A)}$ where $\mathrm{Aut}(A)$ is equipped with the topology 
of pointwise norm convergence. For automorphisms $\alpha$ and $\beta$ of $A$, we say that 
$\alpha$ is \textit{outer conjugate to} $\beta$ if there exist an automorphism $\theta$ of $A$ and 
an element $u$ in $\mathcal{U}(M(A))$ such that 
$$
\alpha= \mathrm{Ad}(u)\circ \theta \circ \beta \circ \theta^{-1}.
$$
It is easy to see that if $\alpha$ is outer conjugate to $\beta$, then 
the crossed product C$^*$-algebra $A\rtimes_{\alpha}\mathbb{Z}$ is isomorphic to 
$A\rtimes_{\beta}\mathbb{Z}$. However, the fixed point algebras 
$A^{\alpha}:=\{x\in A\; |\; \alpha(x)=x \}$ and $A^{\beta}$ need 
not be isomorphic. 

We say that a C$^*$-algebra $A$ is monotracial if $A$ has a unique tracial state and admits no 
unbounded traces. If $A$ is monotracial, we denote by $\tau_{A}$ the unique tracial state on $A$ 
unless otherwise specified. Every tracial state $\tau$ on $A$ extends uniquely to a tracial state on 
$M(A)$ (see, for example, \cite[Lemma 2.10]{Tik}).  
We denote it by the same symbol $\tau$ for simplicity. 
For any tracial state $\tau$ on $A$, define the dimension function $d_{\tau}$ on $A_{+}$ by 
$d_{\tau}(a)=\lim_{n\to\infty}\tau(a^{1/n})$ for any $a\in A_{+}$. 
Let $A$ be a monotracial C$^*$-algebra. 
An element $u$ in $\mathcal{U}(M(A))$ is said to be a \textit{Haar unitary element} if 
$\tau_A(u^n)=0$ for any $n\in \mathbb{Z}\setminus\{0\}$. If $u$ is a Haar unitary element, then
$\tau_A(f(u))>0$ for any $f\in C(\mathbb{T})_{+}\setminus \{0\}$, where $\mathbb{T}$ denotes the  
set of complex numbers of modulus one. 

\subsection{Invariants of aperiodic automorphisms of monotracial C$^*$-algebras}
\label{subsec:invariant}

Let $A$ be a simple separable non-type I monotracial C$^*$-algebra, and 
let $\alpha$ and $\beta$ be automorphisms of $A$. 
We denote by $(\pi_{\tau_A}, H_{\tau_A})$ the Gelfand--Naimark--Segal (GNS) representation of $\tau_A$. 
Then $M:=\pi_{\tau_A}(A)^{''}$ is a II$_1$ factor, and $\alpha$ extends uniquely to an automorphism 
$\tilde{\alpha}$ of $M$. It is easy to see that if $\alpha$ is outer conjugate to $\beta$, 
then $\tilde{\alpha}$ is outer conjugate to $\tilde{\beta}$. 
Define  
$$
N(\tilde{\alpha}):=\left\{\begin{array}{ll}
\min\{n\in\mathbb{N}\; |\; \tilde{\alpha}^n\in \mathrm{Inn}(M) \} & \text{if } 
\tilde{\alpha}^n \in   \mathrm{Inn}(M) \text{ for some } n\in\mathbb{N}, 
\\
\infty & \text{otherwise.}
\end{array}
\right.
$$
Put $\Lambda(\tilde{\alpha}):= 1$ if $N(\tilde{\alpha})=\infty$. 
If $N(\tilde{\alpha})<\infty$, then there exists an element $U$ in $\mathcal{U}(M)$ such that 
$\tilde{\alpha}^{N(\tilde{\alpha})}=\mathrm{Ad}(U)$. Since $M$ is a factor, there exists 
a complex number $\Lambda(\tilde{\alpha})$ of modulus one such that 
$\tilde{\alpha}(U)=\Lambda(\tilde{\alpha})U$ and $\Lambda(\tilde{\alpha})^{N(\tilde{\alpha})}=1$. 
It is easy to see that $(N(\tilde{\alpha}), \Lambda(\tilde{\alpha}))$ is an outer conjugacy invariant 
of $\tilde{\alpha}$ (and hence of $\alpha$). 
Conversely, Connes showed that if $M$ is the injective II$_1$ 
factor and $(N(\tilde{\alpha}), \Lambda(\tilde{\alpha}))=(N(\tilde{\beta}), \Lambda(\tilde{\beta}))$, 
then $\tilde{\alpha}$ is outer conjugate to $\tilde{\beta}$ in \cite{C2} and \cite{C3}. 
We use the following proposition in the proof of the classification theorem in this paper. 

\begin{pro}\label{pro:char-group}
Let $A$ be a simple separable monotracial C$^*$-algebra. 
Assume that $A$ is isomorphic to $A\otimes A$. If $\alpha$ is an aperiodic automorphism of 
$A$, then there exists an aperiodic automorphism $\gamma$ of $A$ such that 
$N(\tilde{\gamma})=N(\tilde{\alpha})$ and $\Lambda(\tilde{\gamma})=
\overline{\Lambda (\tilde{\alpha})}$.  
\end{pro}
\begin{proof}
Since the case where $N(\tilde{\alpha})=\infty$ is trivial, we may assume that  
$N(\tilde{\alpha})<\infty$. Put 
$$
G:= \{\Lambda (\tilde{\beta})\in\mathbb{T} \; | \; \beta \text{ is an aperiodic automorphism of } 
A \text{ such that } N(\tilde{\beta})=N(\tilde{\alpha})\}.
$$
Let $\beta_1$ and $\beta_2$ be aperiodic automorphisms of $A$ satisfying 
$N(\tilde{\beta}_1)=N(\tilde{\beta}_2)=N(\tilde{\alpha})$. Then 
$\beta_1\otimes \beta_2$ is an aperiodic automorphism of $A\otimes A\cong A$ such that 
$N(\tilde{\beta}_1\otimes \tilde{\beta}_2)=N(\tilde{\alpha})$. 
Since we also have $\Lambda (\tilde{\beta}_1\otimes \tilde{\beta}_2)=\Lambda (\tilde{\beta}_1)
\Lambda (\tilde{\beta}_2)$, $G$ is closed under multiplication. 
Furthermore, $G$ is a non-empty subset of a finite group $\{\lambda\in\mathbb{T}\; |\; 
\lambda^{N(\tilde{\alpha})}=1\}$. Hence $G$ is a subgroup. In particular, we have 
$\overline{\Lambda (\tilde{\alpha})}=\Lambda (\tilde{\alpha})^{-1}\in G$, which implies 
the conclusion. 
\end{proof}

\subsection{Central sequence C$^*$-algebras}

We recall the definitions and notation of Kirchberg's central sequence C$^*$-algebras \cite{Kir2}. 
Fix a free ultrafilter $\omega$ on $\mathbb{N}$. For a C$^*$-algebra $A$, put 
$$
A^{\omega}:= \ell^{\infty}(\mathbb{N}, A)/ \{\{x_n\}_{n\in\mathbb{N}}\; |\; \lim_{n\to\omega} \|x_n\| =0 \}.
$$
We denote by $(x_n)_n$ a representative of an element in $A^{\omega}$. 
Every automorphism $\alpha$ of $A$ induces an automorphism of $A^{\omega}$. 
We denote it by the same symbol $\alpha$ for simplicity. For a tracial state $\tau_A$ on $A$, define 
a tracial state $\tau_{A, \omega}$ on $A^{\omega}$ by $\tau_{A, \omega}((x_n)_n)=\lim_{n\to\omega}
\tau_A(x_n)$ for any $(x_n)_n\in A^{\omega}$. 
Let $B$ be a C$^*$-subalgebra of $A$. We identify $A$ and $B$ with C$^*$-subalgebras of 
$A^{\omega}$ consisting of equivalence classes of constant sequences. 
Put 
$$
\mathrm{Ann}(B, A^{\omega}):= \{(x_n)_n \in A^{\omega}\cap B^{\prime}\; |\; 
(x_n)_n b =0\text{ for any } b\in B\}
$$
and define 
$$
F(A):=(A^{\omega}\cap A^{\prime})/\mathrm{Ann}(A, A^{\omega}). 
$$
Note that $((A^{\sim})^{\omega}\cap A^{\prime})/\mathrm{Ann}(A, (A^\sim)^{\omega})$ 
and  $(M(A)^{\omega}\cap A^{\prime})/\mathrm{Ann}(A, M(A)^{\omega})$ are isomorphic to 
$F(A)$ (see \cite[Section 2.2]{Na2}). 
Every automorphism $\alpha$ of $A$ induces an 
automorphism of $F(A)$. We also denote it by the same symbol $\alpha$. 
If $A$ is monotracial, then  
$\tau_{A, \omega}$ induces a tracial state on $F(A)$ by \cite[Proposition 2.1]{Na4}. 
We denote it by the same symbol $\tau_{A, \omega}$ for simplicity.  

Let $A$ be a simple separable monotracial C$^*$-algebra. Put
$$
\mathcal{M}^{\omega}(A):=\ell^{\infty}(\mathbb{N}, \pi_{\tau_A}(A)^{''})/
\{\{x_n\}_{n\in\mathbb{N}}\; |\; \lim_{n\to\omega} \|x_n\|_2 =0 \}
$$
where $\| \cdot \|_2$ denotes the $2$-norm induced by $\tau_A$. 
Note that $\mathcal{M}^{\omega}(A)$ coincides with the ultrapower von Neumann algebra of 
$\pi_{\tau_A}(A)^{''}$. Define a homomorphism $\varrho$ from $A^{\omega}$ to 
$\mathcal{M}^{\omega}(A)$ by 
$$
\varrho ((x_n)_n):= (\pi_{\tau_A}(x_n))_n
$$ 
for any $(x_n)_n\in A^{\omega}$. Then $\varrho$ is surjective by the Kaplansky density theorem. 
Note that $\varrho$ extends to a homomorphism from $M(A)^{\omega}$ to 
$\mathcal{M}^{\omega}(A)$ because $\pi_{\tau_A}$ extends to $M(A)$. 
Put 
$$
\mathcal{M}_{\omega}(A):= \mathcal{M}^{\omega}(A)\cap \varrho(A)^{\prime}.
$$
Note that $\mathcal{M}_{\omega}(A)$ is the von Neumann algebraic central sequence algebra 
of $\pi_{\tau_A}(A)^{''}$. 
We see that $\varrho$ induces a surjective homomorphism from 
$F(A)$ to $\mathcal{M}_{\omega}(A)$. More generally, for any separable 
subset $S$ of $F(A)$, $\varrho$ induces a surjective homomorphism from 
$F(A)\cap S^{\prime}$ to $\mathcal{M}_{\omega}(A)\cap\varrho (S)^{\prime}$  
(see the argument following \cite[Proposition 3.7]{Na5}, as well as \cite[Theorem 3.3]{KR} and 
\cite[Theorem 3.1]{MS3}). 
We denote these homomorphisms by the same symbol $\varrho$ for simplicity. 
Note that for any unitary element $V$ in $\mathcal{M}_{\omega}(A)\cap\varrho (S)^{\prime}$, 
we can choose a unitary element $v$ in $F(A)\cap S^{\prime}$ such that $\varrho (v)=V$ 
because $\mathcal{M}_{\omega}(A)\cap\varrho (S)^{\prime}$ is a von Neumann algebra. 
Indeed, there exists a self-adjoint element $T$ in $\mathcal{M}_{\omega}(A)\cap\varrho
(S)^{\prime}$ such that $V=e^{iT}$ by the spectral theorem 
for von Neumann algebras. By surjectivity of $\varrho$, 
there exists a self-adjoint element $x$ in $F(A)\cap S^{\prime}$ such that $\varrho(x)=T$. 
Putting $v:= e^{ix}\in F(A)\cap S^{\prime}$, we obtain $\varrho (v)=V$. 
Every automorphism 
$\alpha$ of $A$ induces automorphisms of $\mathcal{M}^{\omega}(A)$ and $\mathcal{M}_{\omega}(A)$. 
We denote these automorphisms by $\tilde{\alpha}$ as well. 
By definition, we have $\varrho \circ \alpha = \tilde{\alpha}\circ \varrho$. 

\subsection{Approximate intertwining argument}

Szab\'o's approximate cocycle intertwining argument \cite{Sza7} can be regarded as 
an equivariant version of Elliott's approximate intertwining argument \cite{Ell0} (see also \cite{Ell1}).  
One feature of this argument is that 
a cohomology vanishing theorem is needed only at specific stages of the argument, 
unlike in the Evans--Kishimoto intertwining argument \cite{EK}. 
The following proposition is a restricted but simpler variant of 
\cite[Proposition 5.2]{Na5} (see also \cite[Proposition 3.4]{Sza7}). 

\begin{pro}\label{pro:intertwining}
Let $A$ be a separable C$^*$-algebra, and let $\alpha$ and $\beta$ be automorphisms of $A$. 
Assume that there exist increasing sequences $\{F_n\}_{n\in\mathbb{N}}$ and     
$\{G_n\}_{n\in\mathbb{N}}$ of finite subsets of $A$,  a decreasing sequence 
$\{\varepsilon_n\}_{n\in\mathbb{N}}$ of positive numbers, sequences $\{\varphi_n\}_{n\in\mathbb{N}}$ 
 and $\{\psi_n\}_{n\in\mathbb{N}}$ of nondegenerate endomorphisms of $A$, and 
sequences $\{u_n\}_{n\in\mathbb{N}}$ and $\{v_n\}_{n\in\mathbb{N}}$ of unitary elements 
in $A^{\sim}$ 
such that 
\begin{align*}
(\mathrm{i})\; & \| \varphi_n \circ \alpha (a) - \mathrm{Ad}(u_n)\circ \beta \circ \varphi_n (a)\| 
< \varepsilon_n, \; \text{for any } a\in F_n, \\
(\mathrm{ii})\; &  \| \psi_n \circ \beta (b) - \mathrm{Ad}(v_n)\circ \alpha \circ \psi_n (b)\| 
< \varepsilon_n, \; \text{for any } b\in G_n, \\
(\mathrm{iii})\; & \| \psi_n\circ \varphi_n (a)- a\|<\varepsilon_n, \; \text{for any } 
a\in F_n, \\
(\mathrm{iv})\; & \| \varphi_{n+1} \circ \psi_n (b)- b\|<\varepsilon_n, \; \text{for any } 
b\in G_n, \\
(\mathrm{v})\; & \| \psi_n\circ \varphi_n(a)(\psi_n(u_n)v_n-1) \|<\varepsilon_n, \; 
\text{for any } a\in F_n, \\
(\mathrm{vi})\; & \| \varphi_{n+1}\circ \psi_n(b)(\varphi_{n+1}(v_n)u_{n+1}-1) \|<\varepsilon_n, \; 
\text{for any } b\in G_n, \\
(\mathrm{vii})\; & \varphi_n(F_n)\subset G_{n}, \quad \psi_{n}(G_n)\subset F_{n+1}, \\
(\mathrm{viii})\; & F_nv_n \subset F_{n+1}, \quad G_nu_{n+1}\subset G_{n+1},  \\
(\mathrm{ix})\; & \overline{\bigcup_{n=1}^{\infty} F_n}= 
\overline{\bigcup_{n=1}^{\infty}G_n}=A, \quad \sum_{n=1}^{\infty}\varepsilon_n<\infty.
\end{align*}
Then $\alpha$ is outer conjugate to $\beta$. 
\end{pro}
\begin{proof}
By (iii), (iv), (vii), and (ix), Elliott's approximate intertwining argument (see \cite[Proposition 2.3.2]{Ror}) 
implies that there exist automorphisms $\varphi$ and $\psi$ of $A$ 
such that 
$$
\varphi (a)=\lim_{n\to\infty}\varphi_n(a) \quad \text{and} \quad \psi (b)=\lim_{n\to\infty}\psi_n(b)
\eqno{(2.2.1)}
$$
for any $a, b\in A$. 
Let $n_0\in \mathbb{N}$. 
For any $b\in G_{n_0}$ and $n>n_0$, we have 
\begin{align*}
\|bu_{n+1}-bu_{n} \| & = \|b(1-u_nu_{n+1}^*)\|  \\
&=\| b- \varphi_{n+1}\circ \psi_n (bu_n)u_{n+1}^*
+(\varphi_{n+1}\circ \psi_n (bu_n)-bu_n)u_{n+1}^* \| \\
(\mathrm{iv}), (\mathrm{viii})\; 
& < \| b- \varphi_{n+1}\circ \psi_n (bu_n)u_{n+1}^*\| +\varepsilon_n \\
&= \| b- \varphi_{n+1}\circ \psi_n (bu_n)\varphi_{n+1}(v_{n})
\\
&\quad +\varphi_{n+1}\circ \psi_n (bu_n) (\varphi_{n+1}(v_{n})-u_{n+1}^*)\|+\varepsilon_n \\
(\mathrm{vi}), (\mathrm{viii})\;
& < \| b- \varphi_{n+1}(\psi_n(b)\psi_n (u_n)v_{n})\|+2\varepsilon_n \\
& = \| b- \varphi_{n+1}(\psi_n(\varphi_n\circ \psi_{n-1}(b))\psi_n (u_n)v_n) \\
&\quad +\varphi_{n+1}(\psi_n(\varphi_n\circ \psi_{n-1}(b)-b)\psi_n (u_n)v_n)\|+2\varepsilon_n \\
(\mathrm{iv})\; 
& < \| b- \varphi_{n+1}(\psi_n\circ \varphi_n (\psi_{n-1}(b))\psi_n (u_n)v_n)\| +\varepsilon_{n-1}+
2\varepsilon_n \\
&= \| b- \varphi_{n+1}(\psi_n\circ \varphi_n (\psi_{n-1}(b))) \\
&\quad +\varphi_{n+1}(\psi_n\circ \varphi_n (\psi_{n-1}(b))(1-\psi_n (u_n)v_n))\| +\varepsilon_{n-1}+
2\varepsilon_n \\
(\mathrm{v}), (\mathrm{vii})\; 
& < \| b- \varphi_{n+1}(\psi_n\circ \varphi_n (\psi_{n-1}(b)))\| +\varepsilon_{n-1}+3\varepsilon_n \\
& =\| b- \varphi_{n+1}(\psi_n (b))+
\varphi_{n+1}(\psi_n (b- \varphi_n \circ \psi_{n-1}(b)))\| +\varepsilon_{n-1}+3\varepsilon_n \\
(\mathrm{iv})\;  
& < 2\varepsilon_{n-1}+4\varepsilon_n < 6\varepsilon_{n-1}. 
\end{align*}
Hence we see that $\{bu_n\}_{n\in\mathbb{N}}$ is a Cauchy sequence for any $b\in A$ by (ix). 
Let $n_0\in \mathbb{N}$, and let $a\in F_{n_0}$. 
For any $\varepsilon>0$, there exists a natural number $N\geq n_0$ such that 
$$
\| \varphi_n(a)-\varphi(a)\|<\varepsilon, \quad 
\| \varphi_n(\alpha(a))-\varphi(\alpha(a))\|<\varepsilon, \eqno{(2.2.2)}
$$
$$
\|\varphi(\alpha(a))u_{m}-\varphi(\alpha(a))u_{n}\|
<\varepsilon, \eqno{(2.2.3)}
$$
and $\varepsilon_n< \varepsilon$ for any $m, n\geq N$ 
since (2.2.1) holds and $\{\varphi(\alpha(a))u_n\}_{n\in\mathbb{N}}$ is a Cauchy sequence. 
Then we have 
\begin{align*}
\|u_{m}\beta(\varphi (a))&-u_{n}\beta (\varphi (a))\| \\
(2.2.2)\; 
&< \|u_{m}\beta(\varphi_m(a))-u_{n}\beta(\varphi_{n}(a))\|+2\varepsilon \\
(\mathrm{i})\; 
&< \|\varphi_{m}(\alpha (a))u_{m}-\varphi_{n}(\alpha(a))u_{n}\| +2\varepsilon+\varepsilon_{m}
+\varepsilon_n \\
(2.2.2)\; 
&< \|\varphi(\alpha (a))u_{m}-\varphi(\alpha(a))u_{n}\|+4\varepsilon+\varepsilon_{m}
+\varepsilon_n \\
(2.2.3)\;
&< 5\varepsilon+\varepsilon_{m}+\varepsilon_n<7\varepsilon
\end{align*}
for any $m,n\geq N$. Therefore $\{u_{n}\beta(\varphi(a))\}_{n\in\mathbb{N}}$ is a Cauchy sequence. 
Since $\varphi$ and $\beta$ are automorphisms and $\overline{\bigcup_{n\in\mathbb{N}}F_n}=A$ 
by (ix), $\{u_{n}b\}_{n\in\mathbb{N}}$ is a Cauchy sequence for any $b\in A$. 
Consequently, $\{u_n\}_{n\in\mathbb{N}}$ is a strict Cauchy sequence of unitary 
elements in $A^{\sim}$. Since $M(A)$ is strictly complete, 
there exists a unitary element $u$ in $M(A)$ such that $\lim_{n\to \infty}u_{n}b=ub$ for any 
$b\in A$. The rest of the proof is the same as the proof of \cite[Proposition 5.2]{Na5}. 
\end{proof}

\subsection{Small Rohlin towers}

The following lemma and theorem are based on 
Kishimoto's theorem \cite[Theorem 2.1]{K}, Sato's observation \cite[Proposition 2.1]{Sa2}, Sato and 
Matui--Sato's techniques \cite{Sa0, MS, MS2}, 
and Szab\'o's argument for the equivariant property (SI) \cite[Theorem 4.2]{Sza6} 
(see also \cite[Theorem 3.5]{Na5}).  

\begin{lem}\label{lem:equivariant-si}
Let $A$ be a simple separable non-type I nuclear monotracial C$^*$-algebra with strict comparison, 
and let $\alpha$ be an aperiodic automorphism of $A$. If $e$ and $f$ are positive contractions in 
$A^{\omega}\cap A^{\prime}$ satisfying 
$$
\tau_{A, \omega}(e)=0\quad \text{and}\quad \inf_{m\in\mathbb{N}}\tau_{A, \omega}(f^{m})>0,
$$
then there exists an element $r$ in $A^{\omega}\cap A^{\prime}$ such that 
$$
fr=r, \quad r^*r=e, \quad \text{and}\quad r^*\alpha^{k}(r)=0
$$
for any $k\in\mathbb{Z}\setminus\{0\}$. 
\end{lem}
\begin{proof}
By essentially the same argument as in the first part of the proof of 
\cite[Theorem 3.5]{Na5}, we can prove this lemma. We give a sketch of the proof for the reader's 
convenience.  
By \cite[Lemma 3.2 and Proposition 3.3]{Na5}, we obtain sequences $\{r_n\}_{n\in\mathbb{N}}$
of elements in $A^{\omega}$ and $\{b_n\}_{n\in\mathbb{N}}$ of positive elements in 
$A$ of norm one such that  
$$
fr_n=r_n, \quad \lim_{n\to\infty}\| b_nr_n -r_n\|=0, \quad \lim_{n\to\infty}\|r^*_nar_n-ae\|=0, 
$$
and
$$
\lim_{n\to\infty}\| b_n\alpha^{k}(b_n)\|=0
$$
for any $a\in A^{\sim}$ and $k\in\mathbb{Z}\setminus\{0\}$. Note that we have 
$\lim_{n\to \infty}\| r_n^*\alpha^{k}(r_n)\|=0$ for any $k\in\mathbb{Z}\setminus\{0\}$ since   
$\lim_{n\to\infty}\| b_nr_n -r_n\|=0$ and $\lim_{n\to\infty}\| b_n\alpha^{k}(b_n)\|=0$. 
By the diagonal argument, we obtain 
a contraction $r$ in $A^{\omega}$ such that $fr=r$, $r^*ar=ae$, and 
$r^*\alpha^{k}(r)=0$ for any $a\in A^{\sim}$ and $k\in\mathbb{Z}\setminus\{0\}$. 
Since we have $r^*ar=ae$ and $e\in A^{\omega}\cap A^{\prime}$, 
$r$ is an element in $A^{\omega}\cap A^{\prime}$ 
(see the proof of~\cite[Theorem 3.5]{Na5} for details). Hence $r$ is the desired element. 
\end{proof}

\begin{thm}\label{thm:small-rohlin}
Let $A$ be a simple separable non-type I nuclear monotracial C$^*$-algebra with strict comparison, 
and let $\alpha$ be an aperiodic automorphism of $A$. If $p$ is a projection in 
$\ker(\varrho)\subset F(A)$, then there exist a projection $q$ and an element $s$ in 
$F(A)$ such that 
$$
q\alpha^{k}(q)=0, \quad pq=0, \quad s^*s=p, \quad \text{and}\quad ss^*=q
$$
for any $k\in \mathbb{Z}\setminus\{0\}$. 
\end{thm}
\begin{proof}
Let $\{h_n\}_{n\in\mathbb{N}}$ be an approximate unit for $A$. 
Choose a positive contraction $e$ in $A^{\omega}\cap A^{\prime}$ such that $p=[e]$. 
Replacing $e$ with $(h_n^{1/2})_ne(h_n^{1/2})_n$, we may assume that $(h_n)_n\geq e$. 
Put $f:=(h_n)_n-e$. Then $f$ is a positive contraction in $A^{\omega}\cap A^{\prime}$. 
Since $[e]=p$ is a projection in $\ker(\varrho)$ and $[f]=1-p$ is a projection, 
we have $\tau_{A, \omega}(e)=\tau_{A, \omega}(p)=0$ and 
$\inf_{m\in\mathbb{N}}\tau_{A, \omega}(f^m)=\tau_{A, \omega}(1-p)=1>0$. 
Hence there exists an element $r$ in $A^{\omega}\cap A^{\prime}$ such that 
$fr=r$, $r^*r=e$, and $r^*\alpha^{k}(r)=0$ for any $k\in \mathbb{Z}\setminus\{0\}$ by Lemma 
\ref{lem:equivariant-si}. 
Put $q:= [rr^*]$ and $s:=[r]$. Then $q$ and $s$ are the desired elements. 
\end{proof}

\section{Weak stability}\label{sec:weak}

In this section, we introduce a notion of weak stability for 
automorphisms of simple separable monotracial C$^*$-algebras. 
Moreover, we prove an absorption theorem (Theorem~\ref{thm:absorption-trivial-char}) 
for weakly stable automorphisms. 

\begin{Def}
Let $A$ be a simple separable monotracial C$^*$-algebra, and let $\alpha$ be an automorphism 
of $A$. We say that $\alpha$ is \textit{weakly stable} if for any unitary element $u$ in $F(A)$ satisfying  
$\varrho (u)=1$, there exist a unitary element $v$ in $F(A)$ and a complex number 
$\lambda$ of modulus one such that $\lambda u=v\alpha (v^*)$. 
\end{Def} 

Part (3) of the following remark is based on a comment by Sato at a seminar held at Kyushu University. 

\begin{rem}
(1) Usually, an automorphism $\alpha$ of $A$ is said to be \textit{stable} if for any unitary element 
$u$ in $F(A)$, there exists a unitary element $v$ in $F(A)$ such that $u=v\alpha(v^*)$. See, 
for example, the introduction of \cite{HO}. 
One of the key points of weak stability is that we only consider unitary elements $u$ satisfying 
$\varrho(u)=1$. \ \\ 
(2) Based on the works of Connes \cite{C2} and Herman--Ocneanu \cite{HO}, 
Kishimoto showed that if $\alpha$ is an automorphism of a 
UHF algebra $A$ satisfying $N(\tilde{\alpha})=\infty$, then  $\alpha$ is stable 
(see \cite{Kis0, Kis00, Kis3}). 
In particular, $\alpha$ is weakly stable. 
\ \\
(3) Let $\alpha$ be an automorphism of the Jiang--Su algebra $\mathcal{Z}$. 
Then $\alpha$ is not stable. Indeed, let 
$\lambda=e^{2\pi i \theta}$, where $\theta$ is an irrational number. Then $\lambda 1_{F(\mathcal{Z})}$ 
is a unitary element in $F(\mathcal{Z})$. 
Suppose that there exists a unitary element $v$ in $F(\mathcal{Z})$ such that 
$\lambda 1_{F(\mathcal{Z})}=v\alpha(v^*)$. Since $\alpha$ is approximately inner, 
\cite[Lemma 4.3]{M1} implies that there exists a unitary element $w$ in $F(\mathcal{Z})$ such that 
$\alpha(v^*)=wv^*w^*$. Therefore the C$^*$-subalgebra of $F(\mathcal{Z})$ generated by 
$v$ and $w$ is isomorphic to the irrational rotation algebra $A_{\theta}$ 
because we have $vw=\lambda wv$ and $A_{\theta}$ is simple. 
However, $A_{\theta}$ has a non-trivial projection, while $F(\mathcal{Z})$ has no non-trivial projections. 
This is a contradiction. Consequently, no automorphism of $\mathcal{Z}$ is stable. 
On the other hand, Sato showed that if $\alpha$ is an automorphism of $\mathcal{Z}$ satisfying 
$N(\tilde{\alpha})=\infty$, then for any unitary element $u$ in $F(\mathcal{Z})$, 
there exist a unitary element $v$ in $F(\mathcal{Z})$ and a complex number $\lambda$ 
of modulus one such that $\lambda u=v\alpha (v^*)$ (see \cite[Theorem 1.2]{Sa00} and 
\cite[Corollary 5.6]{Sa}). In particular, $\alpha$ is weakly stable.  
Another key point of weak stability is that we allow a scalar multiple $\lambda u$. \ \\
(4) For automorphisms of simple separable monotracial C$^*$-algebras with non-trivial $K_1$-groups, 
it seems necessary to modify the definition of weak stability. However, we do not consider such cases 
in this paper for simplicity. 
\end{rem}

The following proposition is a cohomology vanishing type result for weakly stable automorphisms. 
Note that we adopt the convention that $u\alpha (u)\cdots \alpha^{k}(u)=u$ when $k=0$. 

\begin{pro}\label{pro:cohomology-vanishing}
Let $A$ be a simple separable nuclear monotracial C$^*$-algebra, and let $\alpha$ be a weakly 
stable automorphism of $A$. \ \\
(1) If $N(\tilde{\alpha})<\infty$ and $u$ is a unitary element in $F(A)$ such that 
$$
\varrho (u\alpha (u)\cdots\alpha^{N(\tilde{\alpha})-1}(u))=1,
$$ 
then there exist a unitary element $v$ in $F(A)$ and a complex number $\lambda$ of modulus one 
such that $\lambda u=v\alpha (v^*)$. \ \\
(2) If $N(\tilde{\alpha})=\infty$ and $u$ is a unitary element in $F(A)$, then  there exist a unitary 
element $v$ in $F(A)$ and a complex number $\lambda$ of modulus one such that 
$\lambda u=v\alpha (v^*)$.
\end{pro}
\begin{proof}
(1) Since $A$ is nuclear, $\pi_{\tau_A}(A)^{''}$ is the injective II$_1$ factor. 
Hence $\mathcal{M}_{\omega}(A)$ is a II$_1$ factor. (See, for example, \cite[Theorem XIV.4.18]{Tak3}.) 
Moreover, \cite[Theorem 3.2]{C3} implies that for any $1\leq k \leq N(\tilde{\alpha})-1$, 
$\tilde{\alpha}^{k}$ is an outer automorphism of $\mathcal{M}_{\omega}(A)$. 
Hence $\tilde{\alpha}$ induces an outer action of $\mathbb{Z}_{N(\tilde{\alpha})}$ 
on $\mathcal{M}_{\omega}(A)$. Therefore there exists a unitary element $W$ in 
$\mathcal{M}_{\omega}(A)$ such that $\varrho (u)=W\tilde{\alpha}(W^*)$ by \cite[Corollary 2.6]{C3}. 
The surjectivity of $\varrho$ implies that there exists a unitary element $w$ in $F(A)$ such that 
$\varrho (w)=W$. Since we have $\varrho (w^*u \alpha (w))=1$ and $\alpha$ is weakly stable, 
there exist a unitary element $z$ in $F(A)$ and a complex number $\lambda$ of modulus one 
such that $\lambda w^*u\alpha(w)=z\alpha (z^*)$. Putting $v:= wz$, we obtain the conclusion. 
\ \\
(2) Using \cite[Theorem 2.1.3]{C2} instead of \cite[Corollary 2.6]{C3}, 
we obtain the conclusion by the same argument as in the proof of (1). 
\end{proof}

For a separable C$^*$-algebra $A$, let $\mathrm{GL}(A^{\sim})$ denote the set of invertible elements in 
$A^{\sim}$. Assume that $A\subset \overline{\mathrm{GL}(A^{\sim})}$. 
Then for any unitary element $u$ in $F(A)$, there exists a sequence of unitary elements 
$\{u_n\}_{n\in\mathbb{N}}$ in $A^{\sim}$ such that $u=[(u_n)_n]$ by \cite[Proposition 4.9]{Na2}. 
Therefore the following proposition is an immediate consequence of 
Proposition \ref{pro:cohomology-vanishing}. 
It is also a uniqueness result used in the proof of Theorem \ref{thm:absorption-trivial-char}. 

\begin{pro}\label{pro:uniqueness}
Let $A$ be a simple separable nuclear monotracial C$^*$-algebra satisfying 
$A\subset \overline{\mathrm{GL}(A^{\sim})}$, and let $\alpha$ be a weakly stable automorphism of $A$. 
\ \\
(1) Assume that $N(\tilde{\alpha})<\infty$. For any finite subset $F\subset A$ and $\varepsilon>0$, 
there exist a finite subset $F^{\prime}\subset A$ and $\delta>0$ such that the following holds. 
If $u$ is a unitary element in $A^{\sim}$ such that 
$$
\| [u, x] \| <\delta \quad \text{and} \quad \| 
\pi_{\tau_A}(u\alpha(u)\cdots \alpha^{N(\tilde{\alpha})-1}(u))- 
1_{\pi_{\tau_A}(A)^{''}}\|_2 
< \delta 
$$
for any $x\in F^{\prime}$, then there exist a unitary element $v$ in $A^{\sim}$ and 
a complex number $\lambda$ of modulus one such that 
$$
\| [v, x] \|< \varepsilon \quad \text{and} \quad \| x(\lambda u - v\alpha (v^*)) \|< \varepsilon
$$
for any $x\in F$. \ \\
(2)  Assume that $N(\tilde{\alpha})=\infty$. For any finite subset $F\subset A$ and $\varepsilon>0$, 
there exist a finite subset $F^{\prime}\subset A$ and $\delta>0$ such that the following holds. 
If $u$ is a unitary element in $A^{\sim}$ such that 
$$
\| [u, x] \| <\delta 
$$
for any $x\in F^{\prime}$, then there exist a unitary element $v$ in $A^{\sim}$ and a complex 
number $\lambda $ of modulus one such that 
$$
\| [v, x] \|< \varepsilon \quad \text{and} \quad \| x(\lambda u - v\alpha (v^*)) \|< \varepsilon
$$
for any $x\in F$.
\end{pro}

Next, we consider an existence result used in the proof of Theorem \ref{thm:absorption-trivial-char}. 

\begin{lem}\label{lem:exist}
Let $A$ be a simple separable nuclear monotracial C$^*$-algebra, and let $\gamma$ be an 
automorphism of $A$ with $N(\tilde{\gamma})<\infty$ and $\Lambda (\tilde{\gamma})=1$. 
If $\gamma$ is approximately inner, then there exists an element $v$ 
in $A^{\omega}$ such that 
$$
\gamma (x)= vxv^*, \quad v^*vx=vv^*x=x \quad \text{and} \quad 
\tilde{\gamma}(\varrho(v))=\varrho(v)
$$
for any $x\in A$. 
\end{lem}
\begin{proof}
Since $\gamma$ is approximately inner, 
there exists a sequence $\{u_n\}_{n\in\mathbb{N}}$ of unitary elements in 
$M(A)$ such that $\gamma (x)=\lim_{n\to \infty}u_nxu_n^*$ for any $x\in A$. 
Put $u:=(u_nh_n)_n\in A^{\omega}$, where $\{h_n\}_{n\in\mathbb{N}}$ is an approximate unit for 
$A$. Then we have 
$$
\gamma(x)= uxu^* \quad \text{and} \quad u^*ux=uu^*x=x
$$
for any $x\in A$. 
By  the definition of $N(\tilde{\gamma})$ and the assumption $\Lambda (\tilde{\gamma})=1$, 
there exists a unitary element $W$ in $\pi_{\tau_A}(A)^{''}$ such that 
$\tilde{\gamma}^{N(\tilde{\gamma})}=\mathrm{Ad}(W)$ and $\tilde{\gamma}(W)=W$. 
Since $\varrho (A)$ is strongly dense in $\pi_{\tau_A}(A)^{''}\subseteq \mathcal{M}^{\omega}(A)$ and 
$\gamma (x)=uxu^*$ for any $x\in A$, 
$$
\tilde{\gamma}(T)= \varrho(u)T\varrho (u^*)
$$
for any $T\in \pi_{\tau_A}(A)^{''}$. Hence we have $\varrho(u)W\varrho (u^*)=W$. 
Put $w:= u\gamma(u^*)\in A^{\omega}$. Then we have $w^*wx=ww^*x=x$ and 
$$
wx=u\gamma(u^*)x=u\gamma (u^*\gamma^{-1}(x))=u\gamma(\gamma^{-2}(x)u^*)
=u\gamma^{-1}(x)\gamma(u^*)=xu\gamma (u^*)=xw
$$
for any $x\in A$. Hence $[w]$ is a unitary element in $F(A)$. In particular, $\varrho([w])=\varrho (w)$ 
is a unitary element in $\mathcal{M}_{\omega}(A)$. 
Since we have 
\begin{align*}
\varrho(w)\tilde{\gamma}(\varrho(w))\cdots \tilde{\gamma}^{N(\tilde{\gamma})-1}(\varrho(w))
=\varrho (u)\tilde{\gamma}^{N(\tilde{\gamma})}(\varrho(u)^*) 
=\varrho (u) W\varrho(u^*)W^*=1,
\end{align*} 
\cite[Corollary 2.6]{C3} (see also the proof of Proposition \ref{pro:cohomology-vanishing}) implies that 
there exists a unitary element $V$ in $\mathcal{M}_{\omega}(A)$ such that 
$\varrho (w)= V\tilde{\gamma}(V^*)$. 
Hence we have $V^*\varrho( u)= \tilde{\gamma}(V^*\varrho(u))$. 
By the surjectivity of $\varrho$, there exists an element 
$z$ in $A^{\omega}\cap A^{\prime}$ such that $\varrho (z)=V$ and $[z]$ is a unitary element in 
$F(A)$. Put $v:= z^*u$. 
Then $v$ is the desired element. 
\end{proof}

The following proposition is analogous to \cite[Theorem 4.6]{Na8}. 

\begin{pro}\label{pro:approximately-rep}
Let $A$ be a simple separable nuclear monotracial C$^*$-algebra, and let $\gamma$ be an 
automorphism of $A$ with $N(\tilde{\gamma})<\infty$ and $\Lambda (\tilde{\gamma})=1$.
If $\gamma$ is approximately inner, 
then there exist a unitary element $U$ in $\pi_{\tau_A}(A)^{''}$ and 
an element $u$ in $A^{\omega}$  such that 
$$
\gamma(x)=uxu^*, \quad 
u^*ux=uu^*x=x, 
\quad 
\tilde{\gamma}(\varrho(u))=\varrho(u) \quad 
\text{and} \quad \varrho(u^{N(\tilde{\gamma})})=U
$$ 
for any $x\in A$.  
\end{pro}
\begin{proof}
By the definition of $N(\tilde{\gamma})$ and the assumption $\Lambda (\tilde{\gamma})=1$, 
there exists a unitary element $U$ in $\pi_{\tau_A}(A)^{''}$ such that 
$\tilde{\gamma}^{N(\tilde{\gamma})}=\mathrm{Ad}(U)$ and $\tilde{\gamma}(U)=U$. 
Lemma \ref{lem:exist} implies that there exists an element $v$ in $A^{\omega}$ 
such that $\gamma (x)= vxv^*$, $v^*vx=vv^*x=x$ for any $x\in A$ and 
$\tilde{\gamma}(\varrho(v))=\varrho(v)$. 
It is easy to see that $\varrho(v^{N(\tilde{\gamma})})^*U$ is a unitary element in 
$\mathcal{M}_{\omega}(A)^{\tilde{\gamma}}\cap \{\varrho (v)\}^{\prime}$. 
Take a unitary element $W$ in $\mathcal{M}_{\omega}(A)^{\tilde{\gamma}}\cap \{\varrho (v)\}^{\prime}$ 
such that $W^{N(\tilde{\gamma})}=\varrho(v^{N(\tilde{\gamma})})^*U$. 
Since $\varrho$ is surjective, there exists an element $w$ in $A^{\omega}\cap A^{\prime}$ 
such that $[w]$ is a unitary element in $F(A)$ and $\varrho (w)=W$. 
Put $u:=vw$. Then we have $\gamma (x)=ux u^*$ and $u^*ux=uu^*x=x$ for any $x\in A$. 
Since we have 
$$
\tilde{\gamma}(\varrho(u))=\tilde{\gamma}(\varrho(v)W)=\tilde{\gamma}(\varrho(v))\tilde{\gamma}(W)=
\varrho(v)W=\varrho (u)
$$
and
$$
\varrho(u^{N(\tilde{\gamma})})=\varrho(u)^{N(\tilde{\gamma})}
=\varrho (v)W\varrho(v)W\cdots \varrho(v)W
=\varrho(v^{N(\tilde{\gamma})})W^{N(\tilde{\gamma})}=U,
$$
$u$ is the desired element. 
\end{proof}

The following proposition is analogous to \cite[Theorem 5.4]{Na8}. 

\begin{pro}\label{pro:existence}
Let $A$ and $B$ be simple separable nuclear monotracial C$^*$-algebras satisfying 
$A\otimes B\subset \overline{\mathrm{GL}((A\otimes B)^{\sim})}$, and let 
$\alpha$ and $\gamma$ be automorphisms of $A$ and $B$, respectively. 
Assume that $\gamma$ is an approximately inner automorphism such that 
$N(\tilde{\gamma})<\infty$ and $\Lambda (\tilde{\gamma})=1$. 
Then there exist a unitary element $V$ in $\pi_{\tau_{A\otimes B}}(A\otimes B)^{''}$ and 
unitary elements $u_1$ and $u_2$ in $((A\otimes B)^{\sim})^{\omega}$  such that 
$$ (\alpha\otimes \gamma) (x)= \mathrm{Ad}(u_1) \circ (\alpha \otimes \mathrm{id}_{B})(x), \quad (\alpha\otimes \mathrm{id}_B) (x)= \mathrm{Ad}(u_2) \circ (\alpha \otimes \gamma)(x), $$ for any $x\in A\otimes B$, $$ \varrho(u_1(\alpha \otimes \mathrm{id}_B)(u_1)\cdots (\alpha \otimes \mathrm{id}_B)^{N(\tilde{\gamma})-1}(u_1))=V $$ and $$ \varrho(u_2(\alpha \otimes \gamma)(u_2)\cdots (\alpha \otimes \gamma)^{N(\tilde{\gamma})-1} (u_2))=V^*. 
$$
\end{pro}
\begin{proof}
By Proposition \ref{pro:approximately-rep}, there exist a unitary element 
$U$ in $\pi_{\tau_B}(B)^{''}$ and 
an element $u=(u_n)_n$ in $B^{\omega}$  such that 
$\gamma(b)=ubu^*$, $u^*ub=uu^*b=b$, $\tilde{\gamma}(\varrho(u))=\varrho(u)$  
and $\varrho(u^{N(\tilde{\gamma})})=U$ for any $b\in B$.  
Put 
$$
V:=1_{\pi_{\tau_A}(A)^{''}}\otimes U \in \pi_{\tau_{A\otimes B}}(A\otimes B)^{''}\quad \text{and} \quad  
v:= (h_n\otimes u_n)_n\in (A\otimes B)^{\omega}, 
$$ 
where $\{h_n\}_{n\in\mathbb{N}}$ is an approximate unit for $A$. 
Then we have 
$$
(\alpha\otimes \gamma) (x)=v(\alpha\otimes \mathrm{id}_B)(x)v^*, \quad 
v^*vx=vv^*x=x
$$
for any $x\in A\otimes B$,
$$
\varrho ((\alpha\otimes \mathrm{id}_{B}) (v))=\varrho(v), 
\quad 
\varrho ((\alpha\otimes \gamma) (v))=\varrho(v) \quad \text{and} \quad 
\varrho(v^{N(\tilde{\gamma})})=V.
$$
By the assumption $A\otimes B\subset \overline{\mathrm{GL}((A\otimes B)^{\sim})}$, 
there exists a sequence $\{y_n\}_{n\in\mathbb{N}}$ of invertible elements 
in  $(A\otimes B)^{\sim}$ such that $v=(y_n)_n$. 
Put $u_1:=(y_n(y_n^*y_n)^{-\frac{1}{2}})_n$ and $u_2:=u_1^*$. 
Then we have $u_1x=vx$ and $u_2x=v^*x$ for any $x\in A\otimes B$. Moreover, 
we have $\varrho(u_1)=\varrho(v)$ and $\varrho (u_2)=\varrho(v^*)$. 
Therefore it can be easily checked that $u_1$ and $u_2$ are the desired elements.   
\end{proof}

The following corollary is an existence result used in the proof of Theorem 
\ref{thm:absorption-trivial-char}. 

\begin{cor}\label{cor:existence}
Let $A$ and $B$ be simple separable nuclear monotracial C$^*$-algebras satisfying 
$A\otimes B\subset \overline{\mathrm{GL}((A\otimes B)^{\sim})}$, and let 
$\alpha$ and $\gamma$ be automorphisms of $A$ and $B$, respectively. 
Assume that $\gamma$ is an approximately inner automorphism. \ \\
(1) If $N(\tilde{\gamma})<\infty$ and $\Lambda (\tilde{\gamma})=1$, then there exists 
a unitary element $U$ in $\pi_{\tau_{A\otimes B}}(A\otimes B)^{''}$ such that 
the following hold: \ \\
(i) For any finite subset $F\subset A\otimes B$ and $\varepsilon>0$, there exists a unitary element $u$ 
in $(A\otimes B)^{\sim}$ such that 
$$
\| (\alpha\otimes \gamma) (x)- \mathrm{Ad}(u) \circ (\alpha \otimes \mathrm{id}_{B})(x)\| <\varepsilon 
$$
for any $x\in F$ and 
$$ 
\| \pi_{\tau_{A\otimes B}}(u(\alpha\otimes \mathrm{id}_{B})(u)\cdots 
(\alpha\otimes \mathrm{id}_{B})^{N(\tilde{\gamma})-1}(u))-U\|_2 < \varepsilon.
$$
(ii) For any finite subset $F\subset A\otimes B$ and $\varepsilon>0$, there exists a unitary element $v$ 
in $(A\otimes B)^{\sim}$ such that 
$$
\| (\alpha\otimes \mathrm{id}_B) (x)- \mathrm{Ad}(v) \circ (\alpha \otimes \gamma)(x)\| <\varepsilon 
$$
for any $x\in F$ and 
$$ 
\| \pi_{\tau_{A\otimes B}}(v(\alpha\otimes \gamma)(v)\cdots 
(\alpha\otimes \gamma)^{N(\tilde{\gamma})-1}(v))-U^*\|_2 < \varepsilon.
$$
(2) If $N(\tilde{\gamma})=\infty$, then for any finite subset 
$F\subset A\otimes B$ and $\varepsilon>0$, there exists a unitary element $u$ 
in $(A\otimes B)^{\sim}$ such that 
$$
\| (\alpha\otimes \gamma) (x)- \mathrm{Ad}(u) \circ (\alpha \otimes \mathrm{id}_{B})(x)\|
<\varepsilon 
$$
for any $x\in F$.
\end{cor}
\begin{proof}
Part (1) is an immediate consequence of Proposition \ref{pro:existence}. 
Part (2) follows from the approximate innerness of $\gamma$. 
\end{proof}

The following theorem is the main result of this section.  

\begin{thm}\label{thm:absorption-trivial-char}
Let $A$ and $B$ be simple separable nuclear monotracial C$^*$-algebras satisfying 
$A\otimes B\subset \overline{\mathrm{GL}((A\otimes B)^{\sim})}$, and let 
$\alpha$ and $\gamma$ be automorphisms of $A$ and $B$, respectively. 
Assume that $\alpha\otimes \mathrm{id}_B$ and $\alpha\otimes \gamma$ are weakly stable. 
If $\gamma$ is an approximately inner automorphism satisfying $N(\tilde{\alpha})=N(\tilde{\gamma})$ 
and  $\Lambda (\tilde{\gamma})=1$, then $\alpha\otimes \gamma$ is outer conjugate to 
$\alpha\otimes \mathrm{id}_{B}$. 
\end{thm}
\begin{proof}
Proposition \ref{pro:uniqueness}, Corollary \ref{cor:existence}, and Proposition \ref{pro:intertwining} 
imply the conclusion. 
Indeed, let $\{x_n\}_{n\in\mathbb{N}}$ 
be a dense sequence in $A\otimes B$, and set $\varepsilon_n:=2^{-n}$ for each $n\in\mathbb{N}$. 
First, we consider the case where $N(\tilde{\gamma})<\infty$. Put 
$\beta_1:= \alpha\otimes\gamma$ and $\beta_2:=\alpha\otimes \mathrm{id}_{B}$. 
Note that we have $N(\tilde{\beta}_1)=N(\tilde{\beta}_2)=N(\tilde{\gamma})$. 
Choose a unitary element $U$ in 
$\pi_{\tau_{A\otimes B}}(A\otimes B)^{''}$ as in Corollary \ref{cor:existence}(1). 

Step 1. 
Put $F_1:=\{x_1\}$. Applying Proposition \ref{pro:uniqueness}(1) to $\beta_1$, 
$F_1$, and 
$\varepsilon_1/2$, we obtain $F_1^{\prime}$ and $\delta_1>0$. 
We may assume that $\delta_1<\varepsilon_1/2$. 
Put 
$$
F_1^{\prime\prime}:= F_1\cup F_1^{\prime}\cup \beta_1^{-1}(F_1^{\prime}).
\eqno{(3.9.1)}
$$
Corollary \ref{cor:existence}(1)(i) implies that there exists a unitary element $u_1$ in 
$(A\otimes B)^{\sim}$ 
such that 
$$
\| \beta_1(x)-\mathrm{Ad}(u_1)\circ \beta_2(x)\| 
< \frac{\delta_1}{2N(\tilde{\gamma})} \eqno{(3.9.2)}
$$
for any $x\in F_1^{\prime\prime}$ and 
$$ 
\| \pi_{\tau_{A\otimes B}}(u_1\beta_2(u_1)\cdots 
\beta_2^{N(\tilde{\gamma})-1}(u_1) )-U\|_2 < 
\frac{\delta_1}{2N(\tilde{\gamma})}. \eqno{(3.9.3)}
$$
Put $\varphi_1:=\mathrm{id}_{A\otimes B}$ and $\tilde{u}_1:=u_1$. 
Then (3.9.1) and (3.9.2) imply that $\varphi_1$ is a nondegenerate endomorphism such that 
$$
\| \varphi_1\circ \beta_1 (x)- \mathrm{Ad}(u_1)\circ \beta_2\circ \varphi_1(x) \|
<\frac{\delta_1}{2N(\tilde{\gamma})} <\varepsilon_1
$$
for any $x\in F_1$. 

Step 2. Choose a finite subset $G_1$ of $A\otimes B$ such that 
$$
\varphi_1 (F_1)\cup\beta_1^{-1}(F_1^{\prime}) \subset G_1 \eqno{(3.9.4)}
$$
and
$$
u_1\beta_2(u_1)\cdots \beta_2^{k}(u_1)
\in \{x+\lambda 1\; |\; x\in G_1, \lambda\in\mathbb{C}\}
\eqno{(3.9.5)}
$$
for any $0\leq k\leq N(\tilde{\gamma})-1$. 
Applying Proposition \ref{pro:uniqueness}(1) to $\beta_2$, $G_1$, and 
$\varepsilon_1/2$, we obtain $G_1^{\prime}$ and $\delta_2>0$. We may assume that 
$\delta_2< \delta_1/2$. 
Put
$$
G_1^{\prime\prime}:=G_1\cup G_1^{\prime}\cup\beta_2^{-1}(G_1^{\prime}). \eqno{(3.9.6)}
$$
Corollary \ref{cor:existence}(1)(ii) implies that there exists a unitary element 
$v_1$ in $(A\otimes B)^{\sim}$ such that 
$$
\| \beta_2(x)-\mathrm{Ad}(v_1)\circ \beta_1(x)\| 
< \frac{\delta_2}{2N(\tilde{\gamma})} \eqno{(3.9.7)}
$$
for any $x\in G_1^{\prime\prime}$ and 
$$ 
\| \pi_{\tau_{A\otimes B}}(v_1\beta_1(v_1)\cdots 
\beta_1^{N(\tilde{\gamma})-1}(v_1) )-U^*\|_2 < 
\frac{\delta_2}{2N(\tilde{\gamma})}. \eqno{(3.9.8)}
$$
Put $z_1:= u_1v_1$. Then we have 
\begin{align*}
\| [z_1, x] \|&= \| u_1v_1\beta_1(\beta_1^{-1}(x)) v_1^*u_1^* -\beta_1(
\beta_1^{-1}(x)) \| \\
& =\| u_1(v_1\beta_1(\beta_1^{-1}(x))v_1^*-\beta_2(\beta_1^{-1}(x)))u_1^* \\
&\quad +u_1\beta_2(\beta_1^{-1}(x))u_1^*-\beta_1(\beta_1^{-1}(x))\| \\
(3.9.4), (3.9.6), (3.9.7) \;
& <\|u_1\beta_2(\beta_1^{-1}(x))u_1^*-\beta_1(\beta_1^{-1}(x))\| + \frac{\delta_2}{2N(\tilde{\gamma})} \\
(3.9.1), (3.9.2)\; &
< \frac{\delta_1}{2N(\tilde{\gamma})}+\frac{\delta_2}{2N(\tilde{\gamma})}<\delta_1
\end{align*}
for any $x\in F_1^{\prime}$. 
By (3.9.5), (3.9.6), and (3.9.7), we obtain
\begin{align*}
\| z_1\beta_1&(z_1)\cdots\beta_1^{N(\tilde{\gamma})-1}(z_1) \\
& -u_1\beta_2(u_1)\cdots \beta_2^{N(\tilde{\gamma})-1}(u_1)
v_1\beta_1(v_1)\cdots \beta_1^{N(\tilde{\gamma})-1}(v_1)\|
< \frac{(N(\tilde{\gamma})-1)\delta_2}{2N(\tilde{\gamma})}< \frac{\delta_2}{2}.
\end{align*}
This estimate, (3.9.3), and (3.9.8) imply that 
\begin{align*}
\|\pi_{\tau_{A\otimes B}}(z_1\beta_1&(z_1)\cdots\beta_1^{N(\tilde{\gamma})-1}(z_1))
-1_{\pi_{\tau_{A\otimes B}}(A\otimes B)^{''}}\|_2
<\frac{\delta_2}{2}+\frac{\delta_1}{2N(\tilde{\gamma})}+\frac{\delta_2}{2N(\tilde{\gamma})}
<\delta_1.
\end{align*} 
Therefore, Proposition \ref{pro:uniqueness}(1) implies that there exist a unitary element $w_1$ in 
$(A\otimes B)^{\sim}$ and a complex number $\lambda_1$ of modulus one such that 
$$
\| [w_1, x] \|< \frac{\varepsilon_1}{2} \quad \text{and} \quad 
\| x(\lambda_1 z_1 -w_1\beta_1 (w_1^*)) \|
< \frac{\varepsilon_1}{2} \eqno{(3.9.9)}
$$
for any $x\in F_1$. Put 
$$
\psi_1:= \mathrm{Ad}(w_1^*)\quad\text{and}\quad \tilde{v}_1:=\lambda_1w_1^*v_1\beta_1(w_1). 
$$
Then $\psi_1$ is a nondegenerate endomorphism such that  
$$
\| \psi_1\circ \beta_2 (x) - \mathrm{Ad}(\tilde{v}_1)\circ \beta_1 \circ \psi_1 (x) \| 
=\| \beta_2(x) - v_1\beta_1(x)v_1^*\| <\frac{\delta_2}{2N(\tilde{\gamma})}<\varepsilon_1
$$
for any $x\in G_1$ and
$$
\| \psi_1\circ \varphi_1(x) -x \|=\|w_1^*xw_1 -x\|= \| [x, w_1]\| < \frac{\varepsilon_1}{2}
<\varepsilon_1 
$$
for any $x\in F_1$. Furthermore, we have
\begin{align*}
\|\psi_1\circ \varphi_1 (x) (\psi_1(\tilde{u}_1)\tilde{v}_1-1) \|
&= \| \lambda_1w_1^*xz_1\beta_1(w_1)-w_1^*x w_1\| \\
&=\| x(\lambda_1 z_1- w_1\beta_1(w_1^*))\|<\frac{\varepsilon_1}{2}<\varepsilon_1 
\end{align*}
for any $x\in F_1$. 

Step 3. 
Choose a finite subset $F_2$ of $A\otimes B$ such that 
$$
\{x_2\}\cup G_1\cup \psi_1(G_1) \cup F_1\tilde{v}_1 
\cup\beta_2^{-1}(G_1^{\prime})
\subset F_2 \eqno{(3.9.10)}
$$
and
$$
v_1\beta_1(v_1)\cdots \beta_1^{k}(v_1)
\in \{x+\lambda 1\; |\; x\in F_2, \lambda\in\mathbb{C}\}
\eqno{(3.9.11)}
$$
for any $0\leq k\leq N(\tilde{\gamma})-1$. 
Applying Proposition \ref{pro:uniqueness}(1) to $\beta_1$, $F_2$, and 
$\varepsilon_2/2$, we obtain $F_2^{\prime}$ and $\delta_3>0$. We may assume that 
$\delta_3< \delta_2/2$ and $\delta_3< \varepsilon_2/(2+2\max_{b\in G_1}\| b\|)$. 
Choose a finite subset $F_2^{\prime\prime}$ of $A\otimes B$ such that 
$$
F_2\cup F_2^{\prime}\cup \beta_1^{-1}(F_2^{\prime})\subset F_2^{\prime\prime} \quad
\text{and}\quad
w_1\in \{x+\lambda 1\; |\; x\in F_2^{\prime\prime}, \lambda\in\mathbb{C}\}.  \eqno{(3.9.12)}
$$
Corollary \ref{cor:existence}(1)(i) implies that there exists a unitary element $u_2$ in 
$(A\otimes B)^{\sim}$ such that 
$$
\| \beta_1(x)-\mathrm{Ad}(u_2)\circ \beta_2(x)\| 
< \frac{\delta_3}{2N(\tilde{\gamma})} \eqno{(3.9.13)}
$$
for any $x\in F_2^{\prime\prime}$ and 
$$ 
\| \pi_{\tau_{A\otimes B}}(u_2\beta_2(u_2)\cdots 
\beta_2^{N(\tilde{\gamma})-1}(u_2) )-U\|_2 < 
\frac{\delta_3}{2N(\tilde{\gamma})}. \eqno{(3.9.14)}
$$
Put $z_2:=v_1u_2$.  
Using (3.9.6), (3.9.7), (3.9.10), (3.9.12), and (3.9.13) instead of (3.9.1), (3.9.2), (3.9.4), (3.9.6), and 
(3.9.7), we obtain 
\begin{align*}
\|[z_2, x] \| <  \frac{\delta_2}{2N(\tilde{\gamma})}+\frac{\delta_3}{2N(\tilde{\gamma})}<\delta_2
\end{align*}
for any $x\in G_1^{\prime}$ by the same estimate as in Step 2.  
By (3.9.11), (3.9.12), and (3.9.13), we obtain 
\begin{align*}
\| z_2\beta_2&(z_2)\cdots\beta_2^{N(\tilde{\gamma})-1}(z_2) \\
& -v_1\beta_1(v_1)\cdots \beta_1^{N(\tilde{\gamma})-1}(v_1)
u_2\beta_2(u_2)\cdots \beta_2^{N(\tilde{\gamma})-1}(u_2)\|
< \frac{(N(\tilde{\gamma})-1)\delta_3}{2N(\tilde{\gamma})}< \frac{\delta_3}{2}.
\end{align*}
This estimate, (3.9.8), and (3.9.14) imply that  
\begin{align*}
\|\pi_{\tau_{A\otimes B}}(z_2\beta_2&(z_2)\cdots\beta_2^{N(\tilde{\gamma})-1}(z_2))
-1_{\pi_{\tau_{A\otimes B}}(A\otimes B)^{''}}\|_2
<\frac{\delta_3}{2}+\frac{\delta_2}{2N(\tilde{\gamma})}+\frac{\delta_3}{2N(\tilde{\gamma})}
<\delta_2.
\end{align*}
Therefore, Proposition \ref{pro:uniqueness}(1) implies that there exist a unitary element $w_2$ in 
$(A\otimes B)^{\sim}$ and a complex number $\lambda_2$ of modulus one such that 
$$
\| [w_2, x] \|< \frac{\varepsilon_1}{2} \quad \text{and} \quad 
\| x(\lambda_2 z_2 -w_2\beta_2 (w_2^*)) \|
< \frac{\varepsilon_1}{2} \eqno{(3.9.15)}
$$
for any $x\in G_1$. Put 
$$
\varphi_2:= \mathrm{Ad}(w_2^*w_1)\quad\text{and}\quad 
\tilde{u}_2:=\overline{\lambda_1}\lambda_2w_2^*w_1u_2\beta_2(w_1^*w_2). 
$$
Then $\varphi_2$ is a nondegenerate endomorphism such that  
$$
\| \varphi_2\circ \beta_1 (x) - \mathrm{Ad}(\tilde{u}_2)\circ \beta_2 \circ \varphi_2 (x) \| 
=\| \beta_1(x) - u_2\beta_2(x)u_2^*\| <\frac{\delta_3}{2N(\tilde{\gamma})}<\varepsilon_2
$$
for any $x\in F_2$ and
$$
\| \varphi_2\circ \psi_1(x) -x \|=\|w_2^*xw_2 -x\|= \| [x, w_2]\| < \frac{\varepsilon_1}{2}
<\varepsilon_1 
$$
for any $x\in G_1$. 
Furthermore, we have
\begin{align*}
\|\varphi_2\circ \psi_1 (x) (\varphi_2(\tilde{v}_1)\tilde{u}_2-1) \| 
&= \| \lambda_2w_2^*xv_1\beta_1(w_1)u_2\beta_2(w_1^*w_2)-w_2^*x w_2\| \\
&= \|\lambda_2w_2^*xv_1(\beta_1(w_1)u_2-u_2\beta_2(w_1))\beta_2(w_1^*w_2) \\
&\quad +w_2^*x\lambda_2 z_2\beta_2(w_2)-w_2^*xw_2\| \\ 
(3.9.12), (3.9.13)\; 
&< \| w_2^*x\lambda_2 z_2\beta_2(w_2)-w_2^*xw_2\| + \frac{\|x\|\delta_3}{2N(\tilde{\gamma})} \\
&=\| x(\lambda_2 z_2-w_2\beta_2(w_2^*))\| + \frac{\|x\|\delta_3}{2N(\tilde{\gamma})} \\
(3.9.15)\; 
& <\frac{\varepsilon_1}{2}+ \frac{\|x\|\delta_3}{2N(\tilde{\gamma})}<\varepsilon_1 
\end{align*}
for any $x\in G_1$. 

Step 4. Choose a finite subset $G_2$ of $A\otimes B$ such that 
$$
F_2\cup \varphi_2(F_2) \cup G_1\tilde{u}_2 \cup \beta_1^{-1}(F_2^{\prime})
\subset G_2 \eqno{(3.9.16)}
$$
and
$$
u_2\beta_2(u_2)\cdots \beta_2^{k}(u_2)
\in \{x+\lambda 1\; |\; x\in G_2, \lambda\in\mathbb{C}\}
\eqno{(3.9.17)}
$$
for any $0\leq k\leq N(\tilde{\gamma})-1$. 
Applying Proposition \ref{pro:uniqueness}(1) to $\beta_2$, $G_2$, and $\varepsilon_2/2$, 
we obtain $G_2^{\prime}$ and $\delta_4>0$.
We may assume that $\delta_4< \delta_3/2$ and $\delta_4< \varepsilon_2/(2+2\max_{a\in F_2}\| a\|)$. 
Choose a finite subset $G_2^{\prime\prime}$ of $A\otimes B$ such that
$$
G_2\cup G_2^{\prime}\cup \beta_2^{-1}(G_2^{\prime}) \subset G_2^{\prime\prime} \quad
\text{and}\quad
w_1^*w_2\in \{x+\lambda 1\; |\; x\in G_2^{\prime\prime}, \lambda\in\mathbb{C}\}.  \eqno{(3.9.18)}
$$
Corollary \ref{cor:existence}(1)(ii) implies that there exists a unitary element $v_2$ in 
$(A\otimes B)^{\sim}$ such that 
$$
\| \beta_2(x)-\mathrm{Ad}(v_2)\circ \beta_1(x)\| 
< \frac{\delta_4}{2N(\tilde{\gamma})} \eqno{(3.9.19)}
$$
for any $x\in G_2^{\prime\prime}$ and 
$$ 
\| \pi_{\tau_{A\otimes B}}(v_2\beta_1(v_2)\cdots 
\beta_1^{N(\tilde{\gamma})-1}(v_2) )-U^*\|_2 < 
\frac{\delta_4}{2N(\tilde{\gamma})}. \eqno{(3.9.20)}
$$
Put $z_3:=u_2v_2$. Essentially the same argument as in the previous steps implies that 
there exist a unitary element $w_3$ in $(A\otimes B)^{\sim}$ and a complex number 
$\lambda_3$ of modulus one such that 
$$
\| [w_3, x] \|< \frac{\varepsilon_2}{2} \quad \text{and} \quad 
\| x(\lambda_3 z_3 -w_3\beta_1 (w_3^*)) \|
< \frac{\varepsilon_2}{2} \eqno{(3.9.21)}
$$
for any $x\in F_2$. Put 
$$
\psi_2:= \mathrm{Ad}(w_3^*w_1^*w_2)\quad\text{and}\quad 
\tilde{v}_2:=\lambda_1\overline{\lambda_2}\lambda_3w_3^*w_1^*w_2v_2\beta_1(w_2^*w_1w_3). 
$$
Then $\psi_2$ is a nondegenerate endomorphism such that  
$$
\| \psi_2\circ \beta_2 (x) - \mathrm{Ad}(\tilde{v}_2)\circ \beta_1 \circ \psi_2 (x) \| 
=\| \beta_2(x) - v_2\beta_1(x)v_2^*\| <\frac{\delta_4}{2N(\tilde{\gamma})}<\varepsilon_2
$$
for any $x\in G_2$ and
$$
\| \psi_2\circ \varphi_2(x) -x \|=\|w_3^*xw_3 -x\|= \| [x, w_3]\| < \frac{\varepsilon_2}{2}
<\varepsilon_2 
$$
for any $x\in F_2$. Using (3.9.18), (3.9.19), and (3.9.21) instead of (3.9.12), (3.9.13), and (3.9.15), 
we obtain
$$
\|\psi_2\circ \varphi_2 (x) (\psi_2(\tilde{u}_2)\tilde{v}_2-1) \|
<\frac{\varepsilon_2}{2}+ \frac{\|x\|\delta_4}{2N(\tilde{\gamma})}<\varepsilon_2 
$$
for any $x\in F_2$ by essentially the same estimate as in Step 3. 

Repeating this process, we obtain $\{F_n\}_{n\in\mathbb{N}}$, $\{G_n\}_{n\in\mathbb{N}}$,  
$\{\varphi_n\}_{n\in\mathbb{N}}$, $\{\psi_n\}_{n\in\mathbb{N}}$,  $\{\tilde{u}_n\}_{n\in\mathbb{N}}$, 
and $\{\tilde{v}_n\}_{n\in\mathbb{N}}$ satisfying the assumptions of Proposition \ref{pro:intertwining}. 
Therefore $\beta_1$ is outer conjugate to $\beta_2$ by Proposition \ref{pro:intertwining}. 

Using  Proposition \ref{pro:uniqueness}(2) and Corollary \ref{cor:existence}(2) instead of 
Proposition \ref{pro:uniqueness}(1) and Corollary \ref{cor:existence}(1), 
we can prove the case where $N(\tilde{\gamma})=\infty$ by a simplified version of the argument 
above without the estimates on the $2$-norm. 
\end{proof}

Let $\mathcal{D}$ be a stably finite strongly self-absorbing C$^*$-algebra.
We refer the reader to~\cite{TW} for basic properties of strongly self-absorbing C$^*$-algebras.
Let $A$ be a $\mathcal{D}$-stable C$^*$-algebra.
We say that an automorphism $\alpha$ of $A$ is \textit{$\mathcal{D}$-absorbing}
if $\alpha$ is outer conjugate to $\alpha\otimes \mathrm{id}_{\mathcal{D}}$.
Here we regard $\alpha\otimes \mathrm{id}_{\mathcal{D}}$ as an automorphism of $A$
via an isomorphism from $A\otimes\mathcal{D}$ onto $A$.
Note that this definition does not depend on the choice of the isomorphism.
Furthermore, we say that an automorphism $\alpha$ of a C$^*$-algebra $A$ 
is \textit{UHF-absorbing} if there exists a UHF algebra $B$ of infinite type such that 
$A$ is $B$-stable and $\alpha$ is $B$-absorbing. 
We show a $\mathcal{D}$-absorption result for weakly stable automorphisms. 

\begin{lem}\label{lem:d-absorption}
Let $A$ be a simple separable nuclear monotracial C$^*$-algebra, and let $\alpha$ be 
an approximately inner automorphism of $A$. If $S$ is a separable subset of $F(A)$, 
then there exists a unitary element $u$ in $F(A)$ such that 
$\alpha(x)=uxu^*$ for any $x\in S$ and $\varrho(\alpha(u))=\varrho(u)$. 
\end{lem}
\begin{proof}
By \cite[Lemma 4.3]{M1}, there exists a unitary element $v$ in $F(A)$ such that 
$\alpha (x)=vxv^*$ for any $x\in S\cup\alpha^{-1}(S)$. 
Since we have 
\begin{align*}
v^*\alpha (v)x=v^*\alpha(v\alpha^{-1}(x))=v^*\alpha(xv)=v^*\alpha(x)\alpha(v)=xv^*\alpha (v) 
\end{align*}
for any $x\in S$, $v^*\alpha (v)$ is a unitary element in $F(A)\cap S^{\prime}$. 
First, we consider the case where $N(\tilde{\alpha})<\infty$. 
The same argument as in the proof of Proposition \ref{pro:cohomology-vanishing} implies that 
$\tilde{\alpha}$ induces an outer action of $\mathbb{Z}_{N(\tilde{\alpha})}$
on $\mathcal{M}_{\omega}(A)$. Since we have 
\begin{align*}
\varrho(v^*\alpha (v))\tilde{\alpha}(\varrho(v^*\alpha(v)))\cdots 
&\tilde{\alpha}^{N(\tilde{\alpha})-1}(\varrho(v^*\alpha(v))) \\
&=\varrho (v^*\alpha(v)\alpha(v^*\alpha(v))\cdots 
\alpha^{N(\tilde{\alpha})-1}(v^*\alpha(v))) \\
&=\varrho(v^*\alpha^{N(\tilde{\alpha})}(v)) \\
&=\varrho(v^*)\tilde{\alpha}^{N(\tilde{\alpha})}(\varrho(v)) \\
&=\varrho(v^*)\varrho(v)=1,  
\end{align*}
there exists a unitary element $W$ in $\mathcal{M}_{\omega}(A)\cap \varrho(S)^{\prime}$ 
such that $\varrho(v^*\alpha(v))= W\tilde{\alpha}(W^*)$ 
by \cite[Proposition 7.2]{Oc}. Note that we have $\varrho(v)W=\tilde{\alpha}(\varrho(v)W)$. 
By  the surjectivity of $\varrho$, there exists a unitary element $w$ in 
$F(A)\cap S^{\prime}$ such that $\varrho(w)=W$. Put $u:=vw$. 
Then $u$ is the desired unitary element. For the case where $N(\tilde{\alpha})=\infty$, 
$\tilde{\alpha}$ induces an outer action of $\mathbb{Z}$ on $\mathcal{M}_{\omega}(A)$. 
Hence we obtain the conclusion by essentially the same argument. 
\end{proof}

\begin{pro}\label{pro:d-absorption}
Let $A$ be a simple separable nuclear monotracial C$^*$-algebra and $\mathcal{D}$ 
a stably finite strongly self-absorbing C$^*$-algebra. 
Assume that $A$ is $\mathcal{D}$-stable. 
If $\alpha$ is a weakly stable, approximately inner automorphism of $A$, then $\alpha$ is 
$\mathcal{D}$-absorbing. 
\end{pro}
\begin{proof}
By \cite[Corollary 3.8]{Sza1} (see also \cite{Sza1c}), it is enough to show that 
there exists a unital homomorphism from $\mathcal{D}$ to $F(A)^{\alpha}$. 
Since $A$ is $\mathcal{D}$-stable, there exists a unital homomorphism $\varphi$ from 
$\mathcal{D}$ to $F(A)$. 

First, we consider the case where $N(\tilde{\alpha})<\infty$. 
Lemma \ref{lem:d-absorption} implies that there exists a unitary element $u$ in $F(A)$ 
such that $\alpha(x)=uxu^*$ for any $x\in \varphi (\mathcal{D})$ and $\varrho(\alpha(u))
=\varrho(u)$. Since we have 
\begin{align*}
\varrho(u)^{N(\tilde{\alpha})}\varrho (x)\varrho(u^*)^{N(\tilde{\alpha})}=
\varrho(\alpha^{N(\tilde{\alpha})}(x))=\tilde{\alpha}^{N(\tilde{\alpha})}(\varrho(x))=\varrho(x)
\end{align*}
for any $x\in \varphi (\mathcal{D})$, $\varrho(u)^{N(\tilde{\alpha})}$ is a unitary element 
in $\mathcal{M}_{\omega}(A)^{\tilde{\alpha}}\cap \varrho(\varphi(\mathcal{D})\cup\{u\})^{\prime}$. 
Take a unitary element $W$ in 
$\mathcal{M}_{\omega}(A)^{\tilde{\alpha}}\cap 
\varrho(\varphi(\mathcal{D})\cup\{ u\})^{\prime}$ such that 
$W^{N(\tilde{\alpha})}=\varrho(u)^{N(\tilde{\alpha})}$. 
Since $\varrho$ is surjective, there exists a unitary element $w$ in 
$F(A)\cap (\varphi(\mathcal{D})\cup\{ u\})^{\prime}$ such that $\varrho (w)=W$. 
Put $v:=uw^*$. Then we have $\alpha(x)=vxv^*$ for any $x\in \varphi (\mathcal{D})$ and 
\begin{align*}
\varrho(v^*\alpha(v^*)\cdots \alpha^{N(\tilde{\alpha})-1}(v^*))
&=W\varrho(u)^*\tilde{\alpha}(W)\varrho(\alpha (u))^*\cdots 
\tilde{\alpha}^{N(\tilde{\alpha})-1}(W)\varrho(\alpha^{N(\tilde{\alpha})-1}(u))^* \\
&=W\varrho(u)^*W\varrho(u)^*\cdots W\varrho (u)^*
=W^{N(\tilde{\alpha})}\varrho(u^{N(\tilde{\alpha})})^*=1.
\end{align*}
Since $\alpha$ is weakly stable, Proposition \ref{pro:cohomology-vanishing}(1) implies that there exist 
a unitary element $z$ in $F(A)$ and a complex number $\lambda$ of modulus one such 
that $\lambda v^*=z\alpha(z^*)$. 
Since we have 
$$
\alpha (\varphi (a))=v\varphi (a)v^*=\overline{\lambda} v\varphi (a) \lambda v^*
=\alpha(z)z^*\varphi(a) z\alpha (z^*)
$$ 
for any $a\in\mathcal{D}$, 
$$
\alpha(z^*\varphi(a)z)=z^*\varphi (a)z
$$ 
for any $a\in \mathcal{D}$. 
Therefore $\mathrm{Ad}(z^*)\circ \varphi$ is a unital homomorphism from $\mathcal{D}$ 
to $F(A)^{\alpha}$. 

Using  Proposition \ref{pro:cohomology-vanishing}(2) instead of  
Proposition \ref{pro:cohomology-vanishing}(1), 
we can prove the case where $N(\tilde{\alpha})=\infty$ by a similar but simpler argument. 
\end{proof}

\section{UHF-absorbing automorphisms of $\mathcal{W}$}\label{sec:uhf}

In this section, we show that every aperiodic UHF-absorbing automorphism of 
$\mathcal{W}$ is weakly stable. 
Moreover, we show that every aperiodic UHF-absorbing automorphism $\alpha$ of 
$\mathcal{W}$ is 
$\mathcal{W}$-absorbing, that is, 
$\alpha$ is outer conjugate to $\alpha\otimes \mathrm{id}_{\mathcal{W}}$ 
(via an isomorphism from $\mathcal{W}\otimes\mathcal{W}$ onto $\mathcal{W}$). 
We also show that every aperiodic UHF-absorbing automorphism of a simple unital 
monotracial AF algebra is weakly stable. 

First, we consider some properties of aperiodic automorphisms of simple separable nuclear 
monotracial C$^*$-algebras. 
The following proposition is an immediate corollary of \cite[Propositions 3.11 and 3.12]{Na5}.

\begin{pro}\label{pro:4.1}
Let $A$ be a simple separable non-type I nuclear monotracial C$^*$-algebra with strict comparison, 
and let $\alpha$ be an aperiodic automorphism of $A$. 
Then $F(A)^{\alpha}$ is monotracial and has strict comparison. 
Furthermore, if $a$ and $b$ are positive elements in $F(A)^{\alpha}$ satisfying 
$d_{\tau_{A, \omega}}(a)<d_{\tau_{A, \omega}}(b)$, then there exists an 
element $r$ in $F(A)^{\alpha}$ such that $r^*br=a$. 
\end{pro}

If $B$ is a UHF algebra of infinite type, then there exists a prime number $k$ such that 
$B\otimes \bigotimes_{n\in\mathbb{N}}M_{k}(\mathbb{C})$ is isomorphic to $B$. 
Also, every $B$-stable simple exact monotracial C$^*$-algebra has strict comparison 
(see \cite{Ror2}). 
By essentially the same proof as that of \cite[Proposition 4.2]{Na2} 
(see also \cite[Proposition 4.2]{Na5}), 
using Proposition \ref{pro:4.1} instead of \cite[Proposition 4.1]{Na2}, 
we obtain the following proposition. 

\begin{pro}\label{pro:4.2}
Let $A$ be a simple separable nuclear monotracial C$^*$-algebra, 
and let $\alpha$ be an aperiodic UHF-absorbing automorphism of $A$. Then the following hold. \ \\
(1) There exists a prime number $k$ such that for any $N\in\mathbb{N}$, there exists a 
unital homomorphism from $M_{k^{N}}(\mathbb{C})$ to $F(A)^{\alpha}$. \ \\
(2) For any $\theta\in [0,1]$, there exists a non-zero projection $p$ in 
$F(A)^{\alpha}$ such that $\tau_{A, \omega}(p)=\theta$. \ \\
(3) Let $h$ be a positive element in $F(A)^{\alpha}$ such that 
$d_{\tau_{A, \omega}}(h)>0$. For any $\theta\in [0, d_{\tau_{A, \omega}}(h))$, 
there exists a non-zero projection $p$ in $\overline{hF(A)^{\alpha}h}$ such that 
$\tau_{A, \omega}(p)=\theta$.
\end{pro}

Next, we consider some properties of $F(\mathcal{W})$. 
Using \cite[Theorem 5.2]{Na2}, we obtain the following theorem  
by essentially the same proof as that of \cite[Theorem 5.3]{Na2} (see also \cite[Theorem 5.8]{Na2} 
and \cite[Section 5]{Na3}). 

\begin{thm}\label{thm:normal-classification}
Let $\Omega$ be a compact metrizable space, and let $\varphi$ and $\psi$ be 
unital homomorphisms from $C(\Omega)$ to $F(\mathcal{W})$. 
If $\tau_{\mathcal{W}, \omega}(\varphi (f))=\tau_{\mathcal{W}, \omega}(\psi (f))>0$ for any 
$f\in C(\Omega)_{+}\setminus\{0\}$, then $\varphi$ and $\psi$ are unitarily equivalent.  
\end{thm}

For a subset $S$ of a C$^*$-algebra $A$, let $C^{*}(S)$ denote the 
C$^*$-subalgebra of $A$ generated by $S$. 
For a Lipschitz map $U$ from $[0,1]$ to $\mathcal{U}(M(A))$, we denote by 
$\mathrm{Lip}(U)$ its Lipschitz constant. 
The following theorem is a variant of \cite[Theorem 5.7]{Na2} with a simpler proof. 

\begin{thm}\label{thm:path}
Let $N$ be a normal element in $F(\mathcal{W})$ such that 
$\tau_{\mathcal{W}, \omega}(f(N))>0$ for any $f\in C(\mathrm{Sp}(N))_{+}\setminus\{0\}$. 
If $u$ is a unitary element in $F(\mathcal{W})\cap \{N\}^{\prime}$, then there exists a Lipschitz 
map $U$ from $[0,1]$ to $\mathcal{U}(F(\mathcal{W})\cap\{N\}^{\prime})$ such that
$$
U(0)=1, \quad U(1)=u, \quad \text{and} \quad \mathrm{Lip}(U)\leq 2\pi.
$$ 
\end{thm}
\begin{proof}
Since $\mathcal{M}_{\omega}(\mathcal{W})$ is a II$_1$ factor 
(see the proof of Proposition \ref{pro:cohomology-vanishing}), there exists a self-adjoint element 
$H$ in $\mathcal{M}_{\omega}(\mathcal{W})$ such that $\| H\|=\pi$ and 
$e^{iH}$ is a Haar unitary element. By the surjectivity of $\varrho$, there exists a self-adjoint element 
$[(h_n)_n]\in F(\mathcal{W})$ such that $\varrho ([(h_n)_n])=H$. 
Using \cite[Lemma 3.6]{Na2} and taking a suitable subsequence of $(h_n)_n$ 
(or the reindexation trick), we obtain a self-adjoint element $h$ in 
$F(\mathcal{W})\cap \{N, u\}^{\prime}$ such that $\|h\|=\pi$, $e^{ih}$ is a Haar unitary element, 
and 
$$
\tau_{\mathcal{W}, \omega}(xy)= \tau_{\mathcal{W}, \omega}(x)\tau_{\mathcal{W}, \omega}(y)
\eqno{(4.4.1)}
$$
for any $x\in C^{*}(\{N, u\})$ and $y\in C^*(\{e^{ih}\})$. 
Define unital homomorphisms $\varphi$ and $\psi$ from 
$C(\mathrm{Sp}(N))\otimes C(\mathbb{T})\cong C(\mathrm{Sp}(N)\times \mathbb{T})$ 
to $F(\mathcal{W})$ by 
$$
\varphi (f_1\otimes f_2):= f_1(N)f_2(e^{ih}) \quad \text{and} \quad 
\psi (f_1\otimes f_2):= f_1(N)f_2(ue^{ih})
$$
for any $f_1\in C(\mathrm{Sp}(N))$ and $f_2\in C(\mathbb{T})$. 
Since $e^{ih}$ is a Haar unitary element, the assumption on $N$ and (4.4.1) imply that 
$
\tau_{\mathcal{W}, \omega}(\varphi(f))>0
$
for any $f\in C(\mathrm{Sp}(N)\times \mathbb{T})_{+}\setminus \{0\}$. 
Furthermore, since we have 
$$
\tau_{\mathcal{W}, \omega}(N^mN^{*k}(ue^{ih})^\ell)=\tau_{\mathcal{W}, \omega}(N^m
N^{*k}u^{\ell}e^{i\ell h})
= \tau_{\mathcal{W}, \omega}(N^mN^{*k}u^{\ell})\tau_{\mathcal{W}, \omega}(e^{i\ell h})=0
$$
for any $m, k\in\mathbb{N}\cup \{0\}$ and $\ell\in \mathbb{N}$, 
$$
\tau_{\mathcal{W}, \omega}(N^mN^{*k}(ue^{ih})^\ell)=\tau_{\mathcal{W}, \omega}(N^mN^{*k}(e^{ih})^\ell)
$$
for any $m, k, \ell\in \mathbb{N}\cup \{0\}$. Hence we have 
$\tau_{\mathcal{W}, \omega}(\varphi(f))=
\tau_{\mathcal{W}, \omega}(\psi(f))$ for any $f\in C(\mathrm{Sp}(N)\times \mathbb{T})$. 
By Theorem \ref{thm:normal-classification}, there exists a unitary element $v$ in $F(\mathcal{W})$ 
such that $\varphi (f)= v\psi(f)v^*$ for any $f\in C(\mathrm{Sp}(N)\times \mathbb{T})$. 
In particular, we have 
$$
vNv^*=N\quad \text{and} \quad vue^{ih}v^*=e^{ih}. \eqno{(4.4.2)}
$$
Define a map $U$ from $[0,1]$ to $\mathcal{U}(F(\mathcal{W}))$ by 
$U(t):= e^{itv^*hv}e^{-ith}$ for any $t\in [0,1]$. 
By (4.4.2), we have $vN=Nv$ and 
$u=e^{iv^*hv}e^{-ih}$. Since $h$ is a self-adjoint element in 
$F(\mathcal{W})\cap \{N\}^{\prime}$ such that $\|h\|=\pi$, 
it is easy to see that $U$ is the desired Lipschitz map. 
\end{proof}
 
The following proposition is related to \cite[Proposition 2.8]{Na8} (see also the proof of 
\cite[Theorem 4.6]{Na8}). 

\begin{pro}\label{pro:reduced-to-trace-kernel}
Let $\alpha$ be an aperiodic UHF-absorbing automorphism of $\mathcal{W}$. 
If $u$ is a unitary element in $F(\mathcal{W})$ such that $\varrho (u)=1$, then there exist a 
projection $p$ in $\ker(\varrho)\cap\{u\}^{\prime}$ and a Lipschitz map $U$ from $[0,1]$ to 
$\mathcal{U}(F(\mathcal{W})\cap\{ p\}^{\prime})$ such that 
$$
\alpha(p)=p, \quad  u=1-p+up, \quad  U(0)=1, \quad  
U(1)=u, \quad \text{and}  \quad \mathrm{Lip}(U)\leq 2\pi.
$$ 
\end{pro}
\begin{proof}
By the diagonal argument,  it is enough to show that for any $0<\varepsilon<1$, 
there exist a projection $p_{\varepsilon}$ in $F(\mathcal{W})\cap\{u\}^{\prime}$ and a 
Lipschitz map $U_{\varepsilon}$ from $[0,1]$ to $\mathcal{U}(F(\mathcal{W})
\cap\{ p_{\varepsilon}\}^{\prime})$ 
such that 
$$
\alpha(p_{\varepsilon})=p_{\varepsilon}, \quad  u=1-p_{\varepsilon}
+up_{\varepsilon}, \quad  U_{\varepsilon}(0)=1, \quad  
U_{\varepsilon}(1)=u, \quad \mathrm{Lip}(U_{\varepsilon})\leq 2\pi
$$ 
and 
$$
\tau_{\mathcal{W}, \omega}(p_{\varepsilon})\leq \varepsilon. 
$$ 
Let $0<\varepsilon<1$. 
Using \cite[Lemma 1.4]{Kas} for $C^*(\{\alpha^n(u-1)\; |\; n\in\mathbb{Z}\})$ and the 
diagonal argument, we obtain a positive contraction $h$ in $\ker(\varrho)$ such that 
$\alpha(h)=h$ and $h(u-1)=(u-1)h=u-1$. 
Since we have $d_{\tau_{\mathcal{W}, \omega}}(1-h)=1$, 
Proposition \ref{pro:4.2}(3) implies that there exists a projection $q_{\varepsilon}$ in 
$\overline{(1-h)F(\mathcal{W})^{\alpha}(1-h)}$ such that 
$\tau_{\mathcal{W}, \omega}(q_{\varepsilon})=1-\varepsilon$. 
 Since we have $(1-h)(u-1)=(u-1)(1-h)=0$, 
$q_{\varepsilon}(u-1)=(u-1)q_{\varepsilon}=0$. 
Put $p_{\varepsilon}:= 1-q_{\varepsilon}$. Then we have $u=1-p_{\varepsilon}+up_{\varepsilon}
=1-p_{\varepsilon}+p_{\varepsilon}u$ and $\tau_{\mathcal{W}, \omega}(p_{\varepsilon})=\varepsilon$. 
Therefore we see that $p_{\varepsilon}$ is the desired projection. 
Since $0<\tau_{\mathcal{W}, \omega}(p_{\varepsilon})<1$, applying Theorem \ref{thm:path} 
to $N:=p_{\varepsilon}$, we obtain $U_{\varepsilon}$. 
Hence the proof is complete. 
\end{proof}

For the proof of the weak stability of aperiodic UHF-absorbing automorphisms of $\mathcal{W}$ 
(Theorem \ref{thm:weak-stability}), we use the following lemma instead of Herman--Ocneanu's 
argument (see the proof of \cite[Theorem 1]{HO}). The proof of the following lemma is based on 
a direct proof of the fact that if $S$ is the bilateral shift operator on 
$\ell^2(\mathbb{Z})$ and $\lambda\in\mathbb{T}$, then $\lambda\oplus S$ is approximately 
unitarily equivalent to $1\oplus S$. 
(Note that this fact itself is a corollary of the Weyl--von Neumann--Berg theorem.)

\begin{lem}\label{lem:Opus}
Let $A$ be a unital C$^*$-algebra, let $\alpha$ be an automorphism of $A$, and fix 
$M, N\in\mathbb{N}$ and $\ell>0$. 
Suppose that $p$ is a projection in $A^{\alpha}$ and 
$w$ is a unitary element in $pAp$. 
Furthermore, assume that there exist a finite sequence $w_0=p, w_1, \ldots, w_M=w$ in 
$\mathcal{U}(pAp)$, a projection $q$ in $A$, and an element $s$ in $A$ such that 
$$
\| w_{j+1}-w_{j} \| \leq\frac{\ell}{M}
$$
for any $0\leq j\leq M-1$, 
$$
q\alpha^{k} (q)=0
$$
for any $1\leq k  \leq MN+1$, and 
$$
qp=0, \quad s^*s=p, \quad \text{and} \quad ss^*=q.
$$
Then there exists a unitary element $v$ in $A$ such that 
$$
\| v\alpha (v^*) - (1-p+ w)\| \leq \frac{\ell}{M}+\frac{8M}{\sqrt{N}}.
$$ 
\end{lem}
\begin{proof}
Define a map $S$ from $\mathbb{N}\cup \{0\}$ to $\mathbb{N}\cup \{0\}$ by 
$S(0)=0$ and $S(n)=n+1$ for any $n\in\mathbb{N}$. 
Let $I:=\{0,1,\ldots,MN+1\}$, and let $\{e_{m, n}\}_{m, n=0}^{MN+1}$ be the standard matrix units of 
$B(\ell^2(I))$. Define a homomorphism $\Phi$ from $pAp\otimes B(\ell^2(I))$ to $A$ 
by
\begin{align*}
&\Phi \left(\sum_{m, n=0}^{MN+1}x_{m, n}\otimes e_{m, n}\right) \\
&:= x_{0, 0}+\sum_{n=1}^{MN+1}x_{0, n}\alpha^{n}(s^*)
+\sum_{m=1}^{MN+1}\alpha^{m}(s)x_{m, 0}+\sum_{m,n=1}^{MN+1}
\alpha^{m}(s)x_{m, n}\alpha^n(s^*). 
\end{align*}
It is easy to see that $\Phi$ is a well-defined homomorphism. Furthermore, we have 
$$
\alpha\left(\Phi \left(\sum_{m, n=0}^{MN}x_{m, n}\otimes e_{m, n}\right)\right) 
=\Phi \left(\sum_{m, n=0}^{MN}\alpha(x_{m, n})\otimes e_{S(m), S(n)}\right). \eqno{(4.6.1)}
$$
Put $V_0:=p\otimes e_{0, 0}$ and 
$$
V_{j}:= \frac{1}{\sqrt{N}}\left(p\otimes e_{(j-1)N+1, 0}+ \sum_{k=1}^{N-1}(w_{j}\alpha(w_{j})\cdots 
\alpha^{k-1}(w_{j}))^*\otimes e_{(j-1)N+k+1, 0}\right)
$$
for any $1\leq j \leq M$. Then $V_j^*V_j=p\otimes e_{0,0}$ for any $0\leq j \leq M$ and 
$V_0V_0^*$, $V_1V_1^*$, \ldots, $V_{M}V_{M}^*$ are mutually orthogonal projections. 
By (4.6.1), we have 
$$
\alpha (\Phi(V_0))= \Phi(V_0)=\Phi(V_0(w_0\otimes e_{0,0})) \eqno{(4.6.2)}
$$ 
and 
\begin{align*}
&\| \alpha (\Phi(V_j))- \Phi(V_j(w_j\otimes e_{0,0}))\| \\
&=\frac{1}{\sqrt{N}}\| 
\Phi((\alpha(w_{j})\alpha^2(w_{j})\cdots 
\alpha^{N-1}(w_{j}))^*\otimes e_{jN+1, 0})-\Phi(w_{j}\otimes e_{(j-1)N+1, 0})\| \\
&\leq \frac{2}{\sqrt{N}} \tag*{(4.6.3)}
\end{align*}
for any $1\leq j \leq M$. Put
$$
Z:= \sum_{j=0}^{M-1}V_{j+1}V_{j}^*+ V_0V_{M}^*.
$$
Then $Z^*Z=ZZ^*=\sum_{j=0}^{M}V_{j}V_{j}^*$. 
Since we have 
$$
\| \alpha (\Phi(V_{j+1}V_{j}^*)) - \Phi(V_{j+1}(w_{j+1}w_{j}^*\otimes e_{0, 0})V_{j}^*)\|  
\leq \frac{4}{\sqrt{N}}
$$
for any $1\leq j\leq M-1$, 
$$
\|\alpha (\Phi(V_1V_0^*))-   \Phi(V_{1}(w_{1}w_{0}^*\otimes e_{0, 0})V_{0}^*)\|\leq \frac{2}{\sqrt{N}},
$$
and 
$$ 
\|\alpha (\Phi(V_0V_M^*))-   \Phi(V_0(w^*\otimes e_{0, 0})V_{M}^*)\|\leq \frac{2}{\sqrt{N}}
$$
by (4.6.2) and (4.6.3),
$$
\left\| \alpha (\Phi (Z))-\sum_{j=0}^{M-1}\Phi(V_{j+1}(w_{j+1}w_{j}^*\otimes e_{0, 0})V_{j}^*)-
\Phi(V_0(w^*\otimes e_{0, 0})V_{M}^*) \right\|
\leq \frac{4M}{\sqrt{N}}. \eqno{(4.6.4)}
$$
Similarly, we have 
$$
\| \alpha (\Phi(Z^*Z))- \Phi(Z^*Z) \| \leq \frac{4M}{\sqrt{N}}. \eqno{(4.6.5)} 
$$
Put
$$
v:= \Phi (Z)+ 1- \Phi(Z^*Z).
$$
Then $v$ is a unitary element in $A$. Since we have 
$\| w_{j}w_{j+1}^*-p\| \leq \ell /M$ for any $0\leq j \leq M-1$ by the assumption and
$V_0V_0^*$, $V_1V_1^*$, \ldots, $V_{M}V_{M}^*$ are mutually orthogonal projections, 
$$
\left\|\sum_{j=0}^{M-1}V_{j+1}(w_{j}w_{j+1}^*\otimes e_{0, 0})V_{j+1}^*-   \sum_{j=0}^{M-1}V_{j+1}V_{j+1}^*
\right\|
\leq \frac{\ell}{M}. 
\eqno{(4.6.6)}
$$
Hence we have 
\begin{align*}
&\| v\alpha (v^*)- (1-p+ w)\| \\
&= \| v\alpha(\Phi (Z^*)+ 1- \Phi(Z^*Z))-(1-p+w)\| \\
(4.6.5)\; 
&\leq \|  v
(\alpha(\Phi (Z^*)) +1-\Phi(Z^*Z))-(1-p+w)\|+ \frac{4M}{\sqrt{N}} \\
&= \| (\Phi(Z)+1-\Phi(Z^*Z))\alpha(\Phi(Z^*))- (\Phi(Z^*Z)-p+w)\| + \frac{4M}{\sqrt{N}} \\
(4.6.4)\; 
&\leq \left\| \Phi\left(\sum_{j=0}^{M-1}V_{j+1}(w_{j}w_{j+1}^*\otimes e_{0, 0})V_{j+1}^*\right)+
w- (\Phi(Z^*Z)-p+w)\right\| + \frac{8M}{\sqrt{N}} \\
(4.6.6)\; 
&\leq \left\|\Phi\left(\sum_{j=0}^{M-1}V_{j+1}V_{j+1}^*\right)+p- 
\Phi(Z^*Z)\right\| +\frac{\ell}{M}+ \frac{8M}{\sqrt{N}} =\frac{\ell}{M}+ \frac{8M}{\sqrt{N}}.  
\end{align*}
The proof is complete.
\end{proof}

The following theorem is one of the main theorems in this section. 
\begin{thm}\label{thm:weak-stability}
Let $\alpha$ be an aperiodic UHF-absorbing automorphism of $\mathcal{W}$, and let $u$ be 
a unitary element in $F(\mathcal{W})$ such that $\varrho (u)=1$. Then there exists a unitary element 
$v$ in $F(\mathcal{W})$ such that $u=v\alpha(v^*)$. 
In particular, every aperiodic UHF-absorbing automorphism of $\mathcal{W}$ is weakly stable. 
\end{thm}
\begin{proof}
By the diagonal argument, it is enough to show that for any $\varepsilon>0$, 
there exists a unitary element $v_{\varepsilon}$ in $F(\mathcal{W})$ such that 
$\| u- v_{\varepsilon}\alpha(v_\varepsilon^*)\|< \varepsilon$. 
Let $\varepsilon >0$. Choose a natural number $M$ satisfying $2\pi/M<\varepsilon/2$, and 
choose a natural number $N$ satisfying $8M/\sqrt{N}<\varepsilon/2$. 
By Proposition \ref{pro:reduced-to-trace-kernel}, there exist a 
projection $p$ in $\ker(\varrho)\cap\{u\}^{\prime}$ and a Lipschitz map $U$ from $[0,1]$ to 
$\mathcal{U}(F(\mathcal{W})\cap\{ p\}^{\prime})$ such that 
$$
\alpha(p)=p, \quad  u=1-p+up, \quad  U(0)=1, \quad  
U(1)=u, \quad \text{and}  \quad \mathrm{Lip}(U)\leq 2\pi.
$$ 
Theorem \ref{thm:small-rohlin} implies that there exist 
a projection $q$ and an element $s$ in 
$F(\mathcal{W})$ such that 
$$
q\alpha^{k}(q)=0, \quad pq=0, \quad s^*s=p, \quad \text{and}\quad ss^*=q
$$
for any $k\in \mathbb{Z}\setminus\{0\}$. For any $0\leq j \leq M$, put $w_{j}:=U(j/M)p$. 
Applying Lemma~\ref{lem:Opus} to $A:=F(\mathcal{W})$, $w:=up$, and $\ell:=2\pi$, 
we obtain the desired unitary element $v_{\varepsilon}$. 
\end{proof}

Let $\mathcal{Q}$ be the universal UHF algebra, that is, the UHF algebra with 
$K_0(\mathcal{Q})\cong \mathbb{Q}$. Since $\mathcal{W}$ is $\mathcal{Q}$-stable and every 
automorphism of $\mathcal{W}$ is approximately inner, the following corollary is an 
immediate consequence of Proposition \ref{pro:d-absorption} and Theorem \ref{thm:weak-stability}. 

\begin{cor}\label{cor:Q-absorbing}
If $\alpha$ is an aperiodic UHF-absorbing automorphism of $\mathcal{W}$, then 
$\alpha$ is $\mathcal{Q}$-absorbing. 
\end{cor}

Next, we show that every aperiodic UHF-absorbing automorphism of $\mathcal{W}$ 
is $\mathcal{W}$-absorbing.

\begin{lem}\label{lem:weakly-outer}
Let $N$ be a natural number. Then there exists an aperiodic automorphism $\alpha$ 
of $\mathcal{Z}$ such that $N(\tilde{\alpha})=N$ and $\Lambda(\tilde{\alpha})=1$. 
\end{lem}
\begin{proof}
By \cite[Theorem 2.1]{Ror2}, there exists a unital embedding $\varphi$ 
of $I(2,6,3):=\{f\in C([0,1], M_2(\mathbb{C})\otimes M_{3}(\mathbb{C}))\; |\; 
f(0)\in M_2(\mathbb{C})\otimes\mathbb{C},\;  f(1)\in \mathbb{C}\otimes 
M_{3}(\mathbb{C})\}$ into $\mathcal{Z}$ such that 
$$
\tau_{\mathcal{Z}}(\varphi (f))= \int_{0}^{1}\tau_{M_{6}(\mathbb{C})}(f(t))\; dt
$$
for any $f\in I(2, 6, 3)$. Let $\theta$ be an irrational number. 
For any $n\in\mathbb{N}$, define a continuous function $h_n$ on $[0,1]$
by
$$
h_n(t):= \left\{\begin{array}{cl}
-2^{n}\theta t+\theta & \text{if } t\in [0, 1/2^n]   \\
0 & \text{if } t\in (1/2^n, 1] 
\end{array}
\right.,
$$
and define $f_n, g\in  I(2, 6, 3)$ by 
$$
f_n(t):= \left(\begin{array}{cc}
                       1   &        0    \\ 
                       0   &    e^{2\pi ih_n(t)}   
 \end{array} \right)\otimes
1_{M_{3}(\mathbb{C})}
\quad
\text{and} 
\quad 
g(t):= \left(\begin{array}{cc}
                       0   &    h_1(t)/\theta     \\ 
                       0   &   0  
 \end{array} \right)\otimes
1_{M_{3}(\mathbb{C})}
$$
for any $t\in [0,1]$. Then $\varphi(f_n)$ and $\varphi(g)$ are a unitary element and 
a contraction in $\mathcal{Z}$, respectively.  
Also, we have 
$$
\| \varphi (f_n^k)\varphi(g)\varphi(f_n^k)^*- \varphi (g)\| \geq |1- e^{-2\pi i k\theta}| \eqno{(4.9.1)}
$$
for any $n, k\in\mathbb{N}$.
Since 
\begin{align*}
|1 - \tau_{\mathcal{Z}}(\varphi (f_n))|
&=\left| 1-\int_{0}^{1}\frac{1+e^{2\pi ih_n(t)}}{2}\; dt\right| \\
&\leq \left| 1- \int_{\frac{1}{2^n}}^{1} 1\; dt\right| +\left|
\int_{0}^{\frac{1}{2^n}}\frac{1+e^{2\pi ih_n(t)}}{2}\; dt \right| \\
& \leq \frac{1}{2^n}+\int_{0}^{\frac{1}{2^n}}\left|\frac{1+e^{2\pi ih_n(t)}}{2} \right| \; dt
\leq \frac{1}{2^n}+\frac{1}{2^n}
=\frac{1}{2^{n-1}}
\end{align*}
for any $n\in\mathbb{N}$, we have 
$$
\sum_{n=1}^{\infty}|1 - \tau_{\mathcal{Z}}(\varphi (f_n))|< \infty. \eqno{(4.9.2)}
$$
For any $n\in\mathbb{N}$, put
$$
w_n:= \varphi(f_1)\otimes \varphi(f_2)\otimes \cdots 
\otimes \varphi(f_n)\otimes 1\otimes \cdots \in 
\bigotimes_{n=1}^{\infty}\mathcal{Z}\cong \mathcal{Z},
$$
and define an automorphism $\beta$ of $\mathcal{Z}$ by 
$\beta(x):= \lim_{n\to\infty}w_nxw_n^*$ for any $x\in \mathcal{Z}$.  
Then $\beta$ is an aperiodic automorphism of 
$\mathcal{Z}$. Indeed, for any $n\in\mathbb{N}$, put
$$
x_n:= 
\overbrace{1\otimes \cdots \otimes 1}^{n-1}\otimes \varphi (g) \otimes 1\otimes \cdots \in 
\bigotimes_{n=1}^{\infty}\mathcal{Z}\cong \mathcal{Z}.
$$ 
Then (4.9.1) implies that $(x_n)_n$ is an element in $F(\mathcal{Z})$ such that 
$\|\beta^k ((x_n)_n)-(x_n)_n\|\geq |1- e^{-2\pi i k\theta}|>0$ for any $k\in\mathbb{N}$. 
Hence $\beta^k$ induces a non-trivial automorphism of $F(\mathcal{Z})$ for any $k\in\mathbb{N}$. 
Therefore $\beta$ is an aperiodic automorphism of $\mathcal{Z}$. 
By (4.9.2), we see that $\{\pi_{\tau_{\mathcal{Z}}}(w_n)\}_{n\in\mathbb{N}}$ is a Cauchy sequence 
with respect to the $2$-norm (see, for example, the proof of \cite[Lemma 3.2]{Na6} for details). 
Hence there exists a unitary element $V$ in $\pi_{\tau_{\mathcal{Z}}}(\mathcal{Z})^{''}$ such that 
$\{\pi_{\tau_{\mathcal{Z}}}(w_n)\}_{n\in\mathbb{N}}$ converges to $V$ in the strong-$^*$ topology. 
Furthermore,  we have $\tilde{\beta}=\mathrm{Ad}(V)$ and $\tilde{\beta}(V)=V$. 
Define an automorphism $\gamma$ of $\mathcal{Z}$ by 
$$
\gamma(x_1\otimes x_2\otimes \cdots \otimes x_N):= x_{N}\otimes x_{1}\otimes \cdots \otimes 
x_{N-1}
$$
for any $x_1, x_2, \ldots, x_N\in \mathcal{Z}$, where we identify $\bigotimes_{n=1}^{N}\mathcal{Z}$
with $\mathcal{Z}$.  
Since we have $N(\tilde{\gamma})=N$ and $\gamma^N=\mathrm{id}_{\mathcal{Z}}$, 
it is easy to see that $\alpha:=\beta\otimes \gamma$ is the desired automorphism. 
\end{proof}

We denote by $M_{n^{\infty}}$ the UHF algebra of type $n^{\infty}$, that is, $M_{n^{\infty}}$ is 
isomorphic to the infinite tensor product $\bigotimes_{k\in\mathbb{N}}M_{n}(\mathbb{C})$. 
Considering the norm closure instead of the weak closure, 
we obtain the following lemma by essentially the same proof as that of 
\cite[Proposition 1.6]{C3}.

\begin{lem}[{cf.~\cite[Proposition 1.6]{C3}}]\label{lem:model-connes}
Let $N$ be a natural number, and let $\lambda$ be a complex number such that $\lambda^{N}=1$. 
Then there exist an automorphism $\beta$ of $M_{N^{\infty}}$ and a unitary element $u$ in 
$M_{N^{\infty}}$ such that $\beta^N=\mathrm{Ad}(u)$, $\beta(u)=\lambda u$, and 
$N(\tilde{\beta})=N$. 
\end{lem}

We construct aperiodic automorphisms of $\mathcal{Q}$. 

\begin{pro}\label{pro:model}
Let $N$ be a natural number, and let $\lambda$ be a complex number such that $\lambda^{N}=1$. 
Then there exists an aperiodic automorphism $\gamma$ of $\mathcal{Q}$ such that 
$N(\tilde{\gamma})=N$ and $\Lambda (\tilde{\gamma})=\lambda$. 
\end{pro}
\begin{proof}
Since $\mathcal{Q}$ is isomorphic to $\mathcal{Q}\otimes \mathcal{Z}\otimes M_{N^{\infty}}$, 
$\gamma:=\mathrm{id}_{\mathcal{Q}}\otimes \alpha\otimes \beta$ is the desired automorphism,  
where $\alpha$ and $\beta$ are the automorphisms in Lemma \ref{lem:weakly-outer} and 
Lemma \ref{lem:model-connes}, respectively. 
\end{proof}

The following theorem is one of the main theorems in this section. 

\begin{thm}\label{thm:W-absorbing}
Let $\alpha$ be an aperiodic UHF-absorbing automorphism of $\mathcal{W}$. Then 
there exists an automorphism $\gamma$ of $\mathcal{Q}$ such that 
$\alpha$ is outer conjugate to $\gamma\otimes \mathrm{id}_{\mathcal{W}}$ (via an isomorphism 
from $\mathcal{Q}\otimes \mathcal{W}$ onto $\mathcal{W}$). 
In particular, every aperiodic UHF-absorbing automorphism of $\mathcal{W}$ is 
$\mathcal{W}$-absorbing. 
\end{thm}
\begin{proof}
In the case where $N(\tilde{\alpha})=\infty$, this theorem is an immediate corollary of 
\cite[Theorem 4.3]{Na7}. We consider the case where $N(\tilde{\alpha})<\infty$. 
Corollary \ref{cor:Q-absorbing} implies that $\alpha$ is $\mathcal{Q}$-absorbing. In particular, 
$\alpha$ is outer conjugate to 
$\alpha\otimes \mathrm{id}_{\mathcal{Q}}\otimes \mathrm{id}_{\mathcal{Q}}$. 
By Proposition \ref{pro:model}, there exist aperiodic automorphisms $\beta$ and $\beta^{\prime}$ 
of $\mathcal{Q}$ such that $N(\tilde{\beta})=N(\tilde{\beta}^{\prime})=N(\tilde{\alpha})$, 
$\Lambda(\tilde{\beta})=\Lambda(\tilde{\alpha})$, and $\Lambda (\tilde{\beta}^{\prime})
=\overline{\Lambda(\tilde{\alpha})}$. 
Since we have $\Lambda ( \tilde{\beta}^{\prime}\otimes \tilde{\beta})=1$, 
$\alpha\otimes \beta^{\prime}\otimes \beta\otimes \mathrm{id}_{\mathcal{Q}}$ is outer 
conjugate to $\alpha\otimes \mathrm{id}_{\mathcal{Q}}\otimes \mathrm{id}_{\mathcal{Q}}$ 
by Theorem \ref{thm:absorption-trivial-char} and Theorem \ref{thm:weak-stability}. 
Also, since we have $\Lambda ( \tilde{\alpha}\otimes \tilde{\beta}^{\prime})=1$, 
$\alpha\otimes \beta^{\prime}\otimes \beta\otimes \mathrm{id}_{\mathcal{Q}}$ is outer conjugate 
to $\mathrm{id}_{\mathcal{W}}\otimes \beta\otimes \mathrm{id}_{\mathcal{Q}}$ 
by Theorem \ref{thm:absorption-trivial-char} and Theorem \ref{thm:weak-stability} again.  
Consequently, $\gamma:=\beta\otimes \mathrm{id}_{\mathcal{Q}}$ is the desired automorphism. 
\end{proof}

The following corollary is an immediate consequence of Theorem \ref{thm:W-absorbing}, 
\cite[Proposition 4.2 and Theorem 4.5]{Na5}, and \cite[Theorem 2.3(2)]{Na7}. 
(Note that the proof of \cite[Theorem 2.3(2)]{Na7} can be simplified by an argument 
similar to that in the proof of Theorem \ref{thm:path}.)

\begin{cor}\label{cor:unitary-path-fixed}
Let $\alpha$ be an aperiodic UHF-absorbing automorphism of $\mathcal{W}$. 
If $u$ is a unitary element in $F(\mathcal{W})^{\alpha}$, 
then there exists a Lipschitz 
map $U$ from $[0,1]$ to $\mathcal{U}(F(\mathcal{W})^{\alpha})$ such that
$$
U(0)=1, \quad U(1)=u, \quad \text{and} \quad \mathrm{Lip}(U)\leq 2\pi.
$$ 
\end{cor}

Next, we show that every aperiodic UHF-absorbing automorphism of a simple unital 
monotracial AF algebra is weakly stable. 

\begin{lem}\label{lem:path-af}
Let $A$ be a simple unital monotracial AF algebra, and let $p$ be a projection in $F(A)$. 
If $u$ is a unitary element in $F(A)\cap \{p\}^{\prime}$, then there exists a Lipschitz 
map $U$ from $[0,1]$ to $\mathcal{U}(F(A)\cap\{p\}^{\prime})$ such that
$$
U(0)=1, \quad U(1)=u, \quad \text{and} \quad \mathrm{Lip}(U)\leq \pi.
$$ 
\end{lem}
\begin{proof}
This lemma is well-known to experts. We give a sketch of the proof for the reader's 
convenience. 
Let $(p_n)_n$ and $(u_n)_n$ be representatives of $p$ and $u$ in $F(A)=A^{\omega}\cap A^{\prime}$, 
respectively. Let $\{A_n\}_{n=1}^{\infty}$ be an increasing sequence of 
unital finite-dimensional C$^*$-subalgebras of $A$ such that $\overline{\bigcup_{n=1}^{\infty}A_n}=A$. 
For any $n\in\mathbb{N}$, define a conditional expectation $\Phi_n$ 
from $A$ to $A\cap A_n^{\prime}$ 
by 
$$
\Phi_n (a) := \sum_{d=1}^{N}\sum_{i=1}^{k_d}e_{i, 1}^{(d)}a e_{1, i}^{(d)}
$$
for any $a\in A$, where $\{e^{(d)}_{i,j}\; |\; 1\leq d \leq N, 1\leq i, j\leq k_d\}$ is a system of 
matrix units of $A_n$. 
Replacing $\{A_n\}_{n=1}^{\infty}$ with $\{A_{m_n}\}_{n=1}^{\infty}$ for 
a suitable nondecreasing sequence $\{m_n\}_{n=1}^{\infty}$ of natural numbers with 
$\lim_{n\to\infty}m_n=\infty$ (or the reindexation trick), we have 
$(\Phi_n(p_n))_n=p$ and $(\Phi_n(u_n))_n=u$. 
Perturbing $\Phi_n(p_n)$ and $\Phi_n(u_n)$, we may assume that $p_n$ is a projection in 
$A\cap A_n^{\prime}$ and $u_n$ is a unitary element in $A\cap A_n^{\prime}$ for 
any $n\in\mathbb{N}$. 
Furthermore, $p_nu_np_n$ and $(1-p_n)u_n(1-p_n)$ 
are ``almost unitary'' elements in $p_n(A\cap A_n^{\prime})p_n$ and 
$(1-p_n)(A\cap A_n^{\prime})(1-p_n)$, respectively. 
Hence perturbing $u_n$, we may assume that $u_n=w_n+v_n$ for any $n\in\mathbb{N}$, where 
$w_n$ and $v_n$ are unitary elements in $p_n(A\cap A_n^{\prime})p_n$ and 
$(1-p_n)(A\cap A_n^{\prime})(1-p_n)$, respectively. 
Since $A\cap A_n^{\prime}$ is an AF algebra (it is isomorphic to a direct sum of corners of $A$), 
$p_n(A\cap A_n^{\prime})p_n$ and $(1-p_n)(A\cap A_n^{\prime})(1-p_n)$ are AF algebras. 
Hence for any $n\in\mathbb{N}$, there exist Lipschitz maps $W_n$ and $V_n$ from 
$[0,1]$ to $\mathcal{U}(p_n(A\cap A_n^{\prime})p_n)$ and 
$\mathcal{U}((1-p_n)(A\cap A_n^{\prime})(1-p_n))$, respectively, such that 
$$
W_n(0)=p_n, \quad W_n(1)=w_n,\quad V_n(0)=1-p_n,\quad V_n(1)=v_n, 
$$
and 
$$
\mathrm{Lip}(W_n)\leq \pi +\frac{1}{n}, \quad 
\mathrm{Lip}(V_n)\leq \pi+\frac{1}{n}. 
$$
Define a map $U$ 
from $[0,1]$ to $\mathcal{U}(F(A)\cap\{p\}^{\prime})$ by 
$$
U(t):= (W_n(t)+V_n(t))_n
$$
for any $t\in [0,1]$. Then $U$ is the desired map. 
\end{proof}

Using Lemma \ref{lem:path-af} instead of Theorem \ref{thm:path}, we obtain the following 
proposition by essentially the same proof as that of Proposition \ref{pro:reduced-to-trace-kernel}.

\begin{pro}\label{pro:reduced-to-trace-kernel-af}
Let $A$ be a simple unital monotracial AF algebra, and let 
$\alpha$ be an aperiodic UHF-absorbing automorphism of $A$. 
If $u$ is a unitary element in $F(A)$ such that $\varrho (u)=1$, then there exist a 
projection $p$ in $\ker(\varrho)\cap\{u\}^{\prime}$ and a Lipschitz map $U$ from $[0,1]$ to 
$\mathcal{U}(F(A)\cap\{ p\}^{\prime})$ such that 
$$
\alpha(p)=p, \quad  u=1-p+up, \quad  U(0)=1, \quad  
U(1)=u, \quad \text{and}  \quad \mathrm{Lip}(U)\leq \pi.
$$ 
\end{pro}

Using Proposition \ref{pro:reduced-to-trace-kernel-af} instead of 
Proposition \ref{pro:reduced-to-trace-kernel}, 
we obtain the following 
theorem by essentially the same proof as that of Theorem \ref{thm:weak-stability}. 

\begin{thm}\label{thm:weak-stability-af}
Let $A$ be a simple unital monotracial AF algebra, and let 
$\alpha$ be an aperiodic UHF-absorbing automorphism of $A$. 
If $u$ is a unitary element in $F(A)$ such that $\varrho (u)=1$, then there exists a unitary element 
$v$ in $F(A)$ such that $u=v\alpha(v^*)$. 
In particular, every aperiodic UHF-absorbing automorphism of $A$ is weakly stable. 
\end{thm}

\section{Classification}\label{sec:classification}

In this section, we show that every aperiodic automorphism of $\mathcal{W}$ is weakly stable. 
The proof of this result is based on the idea of the proof of \cite[Theorem 5.3]{Sa} (see also 
\cite{MS1, W}). We also completely classify aperiodic automorphisms of $\mathcal{W}$ 
up to outer conjugacy. Furthermore, we classify aperiodic 
UHF-absorbing automorphisms of UHF algebras of infinite type up to outer conjugacy. 

Let $Z_{2^{\infty}, 3^{\infty}}:=\{f\in C([0,1], M_{2^{\infty}}\otimes M_{3^{\infty}})\; |\; 
f(0)\in  M_{2^{\infty}}\otimes\mathbb{C}, \; f(1)\in \mathbb{C}\otimes M_{3^{\infty}} \}$. 

\begin{lem}\label{lem:key}
Let $\alpha$ be an aperiodic automorphism of $\mathcal{W}$. 
If $(u_n)_n$ is an element in $\mathcal{W}^{\omega}\cap \mathcal{W}^{\prime}$ such that 
$[(u_n)_n]\in \mathcal{U}(F(\mathcal{W}))$ and $\varrho([(u_n)_n])=1$, then there exists a 
unitary element $z$ in $F(\mathcal{W}\otimes Z_{2^{\infty}, 3^{\infty}})$ such that 
$[(u_n\otimes 1_{ Z_{2^{\infty}, 3^{\infty}}})_n]
=z(\alpha\otimes\mathrm{id}_{ Z_{2^{\infty}, 3^{\infty}}})(z^*)$. 
\end{lem}
\begin{proof}
In this proof, we identify $C([0,1], \mathcal{W}\otimes M_{2^{\infty}}\otimes 
M_{3^{\infty}})$ with $\mathcal{W}\otimes C([0,1])\otimes M_{2^{\infty}}\otimes 
M_{3^{\infty}}$ for notational convenience. 
We also identify $\mathcal{W}\otimes Z_{2^{\infty}, 3^{\infty}}$ with 
$\{f\in C([0,1], \mathcal{W}\otimes M_{2^{\infty}}\otimes M_{3^{\infty}})\; |\; 
f(0)\in \mathcal{W}\otimes M_{2^{\infty}}\otimes\mathbb{C}, 
\; f(1)\in \mathcal{W}\otimes \mathbb{C}\otimes M_{3^{\infty}} \}$. 
Define a homomorphism 
$\Phi_1$ from  $\mathcal{W}\otimes M_{2^{\infty}}$ to $\mathcal{W}\otimes 
M_{2^{\infty}}\otimes M_{3^{\infty}}$ by 
$$
\Phi_1(x\otimes a):= x\otimes a\otimes 1_{M_{3^{\infty}}}
$$ 
for any $x\in \mathcal{W}$ and $a\in M_{2^{\infty}}$, and define a homomorphism $\Phi_2$ from 
$\mathcal{W}\otimes M_{3^{\infty}}$ to   $\mathcal{W}\otimes M_{2^{\infty}}\otimes M_{3^{\infty}}$ by 
$$
\Phi_2(x\otimes b):=x\otimes 1_{M_{2^{\infty}}}\otimes b
$$ 
for any $x\in \mathcal{W}$ and $b\in M_{3^{\infty}}$. 
It is easy to see that $\Phi_1$ and $\Phi_2$ induce unital homomorphisms from 
$F(\mathcal{W}\otimes M_{2^{\infty}})$ and $F(\mathcal{W}\otimes M_{3^{\infty}})$ to
$F(\mathcal{W}\otimes  M_{2^{\infty}}\otimes  M_{3^{\infty}})$, respectively. 
We denote these induced homomorphisms by the same symbols. 
Note that we also have  
$$
\Phi_1([(u_n\otimes 1_{M_{2^{\infty}}})_n])=\Phi_2([(u_n\otimes 1_{M_{3^{\infty}}})_n])
=[(u_n\otimes 1_{M_{2^{\infty}}}\otimes 1_{M_{3^{\infty}}})_n], \eqno{(5.1.1)}
$$
$$
(\alpha\otimes \mathrm{id}_{M_{6^{\infty}}}) \circ \Phi_1=
\Phi_1\circ (\alpha\otimes \mathrm{id}_{M_{2^{\infty}}}), \eqno{(5.1.2)}
$$
and 
$$ 
 (\alpha\otimes \mathrm{id}_{M_{6^{\infty}}}) \circ \Phi_2=
\Phi_2 \circ (\alpha\otimes \mathrm{id}_{M_{3^{\infty}}}). \eqno{(5.1.3)}
$$ 
By Theorem \ref{thm:weak-stability}, there exist unitary elements $v^{(1)}$ in 
$F(\mathcal{W}\otimes M_{2^{\infty}})$ and 
$v^{(2)}$ in $F(\mathcal{W}\otimes M_{3^{\infty}})$ such that 
$$
[(u_n\otimes 1_{M_{2^{\infty}}})_n]=v^{(1)}(\alpha\otimes \mathrm{id}_{M_{2^{\infty}}}) (v^{(1)})^*
\eqno{(5.1.4)}
$$
and
$$
[(u_n\otimes 1_{M_{3^{\infty}}})_n]=v^{(2)}(\alpha\otimes \mathrm{id}_{M_{3^{\infty}}}) (v^{(2)})^*.
\eqno{(5.1.5)}
$$
Since we have  
\begin{align*}
(\alpha\otimes \mathrm{id}_{M_{6^{\infty}}})&(\Phi_1(v^{(1)})^{*}\Phi_2(v^{(2)})) \\
&=(\alpha\otimes \mathrm{id}_{M_{6^{\infty}}})(\Phi_1(v^{(1)})^{*})
(\alpha\otimes \mathrm{id}_{M_{6^{\infty}}})
(\Phi_2(v^{(2)})) \\
(5.1.2), (5.1.3)\; &
=\Phi_1((\alpha\otimes\mathrm{id}_{M_{2^{\infty}}})(v^{(1)})^*)
\Phi_2((\alpha\otimes\mathrm{id}_{M_{3^{\infty}}})(v^{(2)})) \\
(5.1.4), (5.1.5)\; 
&=\Phi_1(v^{(1)*}[(u_n\otimes 1_{M_{2^{\infty}}})_n])\Phi_2([(u_n^*\otimes 1_{M_{3^{\infty}}})_n]v^{(2)})  \\
(5.1.1)\; 
&= \Phi_1(v^{(1)})^{*}\Phi_2(v^{(2)}),
\end{align*}
$\Phi_1(v^{(1)})^{*}\Phi_2(v^{(2)})$ is a unitary element in 
$F(\mathcal{W}\otimes M_{2^{\infty}}\otimes M_{3^{\infty}})^{\alpha\otimes\mathrm{id}_{M_{6^{\infty}}}}$.
Corollary \ref{cor:unitary-path-fixed} implies that there exist self-adjoint elements 
$(h^{(1)}_n)_n$, $(h^{(2)}_n)_n$, $(h^{(3)}_n)_n$, and $(h_n^{(4)})_n$ 
in $(\mathcal{W}\otimes M_{2^{\infty}}\otimes M_{3^{\infty}})^{\omega}
\cap(\mathcal{W}\otimes M_{2^{\infty}}\otimes M_{3^{\infty}})^{\prime}$ 
such that 
$$
\Phi_1(v^{(1)})^{*}\Phi_2(v^{(2)})=[(e^{ih_n^{(1)}}e^{ih_n^{(2)}}e^{ih_n^{(3)}}e^{ih_n^{(4)}})_n],
$$
$$
[(h^{(j)}_n)_n] \in 
F(\mathcal{W}\otimes M_{2^{\infty}}\otimes M_{3^{\infty}})^{\alpha\otimes\mathrm{id}_{M_{6^{\infty}}}}
\eqno{(5.1.6)}
$$
for any $j=1,2,3,4$. 
For any $n\in\mathbb{N}$, define a Lipschitz map $E_n$ from $[0,1]$ to 
$\mathcal{U}((\mathcal{W}\otimes M_{2^{\infty}}\otimes M_{3^{\infty}})^{\sim})$ 
by $E_n(t):=  e^{ith_n^{(1)}}e^{ith_n^{(2)}}e^{ith_n^{(3)}}e^{ith_n^{(4)}}$ for any $t\in [0,1]$. 
Note that we may assume that there exists a positive number $L$ such that 
$\mathrm{Lip}(E_n)\leq L$ for any $n\in\mathbb{N}$. 
Choose elements $(v^{(1)}_n)_n$ in 
$(\mathcal{W}\otimes M_{2^{\infty}}\otimes \mathbb{C})^{\omega}
\cap(\mathcal{W}\otimes M_{2^{\infty}}\otimes M_{3^{\infty}})^{\prime}$
and
$(v_n^{(2)})_n$ in 
$(\mathcal{W}\otimes \mathbb{C}\otimes M_{3^{\infty}})^{\omega}
\cap(\mathcal{W}\otimes M_{2^{\infty}}\otimes M_{3^{\infty}})^{\prime}$
such that $[(v^{(1)}_n)_n]=\Phi_1(v^{(1)})$ and $[(v^{(2)}_n)_n]=\Phi_2(v^{(2)})$.
We may assume that $v_n^{(1)}$ and $v_n^{(2)}$ are contractions for any 
$n\in\mathbb{N}$. By (5.1.1), (5.1.2), and (5.1.4), we have 
$$
[(v^{(1)}_n)_n](\alpha\otimes \mathrm{id}_{M_{6^{\infty}}})([(v_n^{(1)*})_n])
=[(u_n\otimes 1_{M_{2^{\infty}}}\otimes 1_{M_{3^{\infty}}})_n]. \eqno{(5.1.7)}
$$
For any $n\in\mathbb{N}$, define a continuous map $w_n$ from $[0,1]$ to 
$\mathcal{W}\otimes M_{2^{\infty}}\otimes M_{3^{\infty}}$ by 
$$
w_n(t):= v^{(1)}_nE_n(t)
$$
for any $t\in [0,1]$. 
Then $[(w_n(t))_n]$ is a unitary element in $F(\mathcal{W}\otimes M_{2^{\infty}}\otimes M_{3^{\infty}})$
for any $t\in [0,1]$ and we have 
$$
[(w_n(1))_n]= [(v_n^{(2)})_n] \eqno{(5.1.8)}
$$
and 
$$
\| w_n(t)-w_n(s)\|\leq L |t-s| \eqno{(5.1.9)}
$$ 
for any $t, s\in [0,1]$ and 
$n\in\mathbb{N}$.  
Since we have 
\begin{align*}
&\lim_{n\to\omega}
\sup_{t\in [0,1]}\| (w_{n}(t)(\alpha\otimes \mathrm{id}_{M_{6^{\infty}}})(w_n(t)^*)
-u_n\otimes 1_{M_{2^{\infty}}}\otimes 1_{M_{3^{\infty}}} )x\| 
\\
&= \lim_{n\to\omega}\sup_{t\in [0,1]}\|v^{(1)}_nE_n(t)(\alpha\otimes \mathrm{id}_{M_{6^{\infty}}})
(E_n(t)^{*}v_n^{(1)*})x \\
&\qquad\qquad\qquad -(u_n\otimes 1_{M_{2^{\infty}}}\otimes 1_{M_{3^{\infty}}} )x\|\\
&= \lim_{n\to\omega}\sup_{t\in [0,1]}\|v^{(1)}_nE_n(t)(\alpha\otimes \mathrm{id}_{M_{6^{\infty}}})
(E_n(t)^*)x(\alpha\otimes \mathrm{id}_{M_{6^{\infty}}})(v_n^{(1)*}) \\
&\qquad\qquad\qquad 
-(u_n\otimes 1_{M_{2^{\infty}}}\otimes 1_{M_{3^{\infty}}} )x\| \\
(5.1.6)\; 
&=\lim_{n\to\omega}\sup_{t\in [0,1]}\| v^{(1)}_nE_n(t)E_n(t)^*x
(\alpha\otimes \mathrm{id}_{M_{6^{\infty}}})(v_n^{(1)*}) \\
&\qquad\qquad\qquad 
-(u_n\otimes 1_{M_{2^{\infty}}}\otimes 1_{M_{3^{\infty}}} )x\|
\\
&= \lim_{n\to\omega}\|v^{(1)}_n(\alpha\otimes \mathrm{id}_{M_{6^{\infty}}})(v_n^{(1)*})x 
-(u_n\otimes 1_{M_{2^{\infty}}}\otimes 1_{M_{3^{\infty}}} )x\| \\
(5.1.7)\; 
&=0
\end{align*}
for any $x\in \mathcal{W}\otimes M_{2^{\infty}}\otimes M_{3^{\infty}}$
and $\mathcal{W}\otimes \mathbb{C}\otimes M_{2^{\infty}}\otimes M_{3^{\infty}}\subset 
\mathcal{W}\otimes C([0,1]) \otimes M_{2^{\infty}}\otimes M_{3^{\infty}}$ is a nondegenerate 
inclusion, $[(w_n)_n]$ is a unitary element in $F(C([0,1], \mathcal{W}\otimes M_{2^{\infty}}
\otimes M_{3^{\infty}}))$ satisfying 
$$
[(u_n\otimes 1_{C([0,1])}\otimes 1_{M_{2^{\infty}}}\otimes 1_{M_{3^{\infty}}})_n] =
[(w_n)_n](\alpha\otimes 
\mathrm{id}_{C([0,1])\otimes M_{6^{\infty}}})([(w_n^*)_n]). \eqno{(5.1.10)}
$$
For any $n\in\mathbb{N}$, define a continuous map $z_n$ from $[0,1]$ to
$\mathcal{W}\otimes M_{2^{\infty}}\otimes M_{3^{\infty}}$ by 
$$
z_n(t):= \left\{\begin{array}{cl}
w_n(t) & \text{if } t\in [0, 1- 1/2^n]   \\
(-2^nt+2^n)w_n(1-1/2^n)+(2^nt-2^n+1)v_n^{(2)} & \text{if } t\in (1-1/2^n, 1] 
\end{array}
\right.
$$
for any $t\in [0,1]$. 
Since we have $z_n(0)=v_n^{(1)}$ and $z_n(1)=v_n^{(2)}$ for any $n\in\mathbb{N}$, 
$(z_n)_n$ can be regarded as an element in 
$ (\mathcal{W}\otimes Z_{2^{\infty}, 3^{\infty}})^{\omega}\cap 
C([0,1], \mathcal{W}\otimes M_{2^{\infty}}\otimes M_{3^{\infty}})^{\prime}
\subset (\mathcal{W}\otimes Z_{2^{\infty}, 3^{\infty}})^{\omega}\cap
(\mathcal{W}\otimes Z_{2^{\infty}, 3^{\infty}})^{\prime}$.
For any $x\in \mathcal{W}$, we have 
\begin{align*}
&\lim_{n\to\omega}\| (z_n - w_n)(x\otimes 1_{C([0,1], M_{2^{\infty}}\otimes M_{3^{\infty}})})\| \\
&=\lim_{n\to\omega}\sup_{t\in [0,1]} \| (z_n(t) -w_n(t))x\| \\
&=\lim_{n\to\omega}\sup_{t\in (1-1/2^n, 1]}\| (z_n(t)-w_n(t))x\| \\
&=\lim_{n\to\omega}\sup_{t\in (1-1/2^n, 1]}\|(z_n(t)-w_n\left(1-1/2^n\right))x
+(w_n(1-1/2^n)-w_n(t))x\| \\
(5.1.9)\; 
&\leq \lim_{n\to\omega}
\left(\sup_{t\in (1-1/2^n, 1]}\| (z_n(t)-w_n(1-1/2^n))x\| +\frac{L\|x\|}{2^{n}}\right) \\
&=\lim_{n\to\omega}\left(\sup_{t\in  (1-1/2^n, 1]}
\|(2^nt-2^n+1)(v_n^{(2)}-w_n(1-1/2^n))x\|+ \frac{L\|x\|}{2^{n}}\right) \\
&\leq \lim_{n\to\omega}\left(\|(v_n^{(2)}-w_n(1-1/2^n))x\|+ \frac{L\|x\|}{2^{n}}\right) \\
(5.1.9)\; 
&\leq \lim_{n\to\omega}\left(\|(v_n^{(2)}-w_n(1))x\| +\frac{L\|x\|}{2^{n-1}}\right) \ \\
(5.1.8)\; 
&=0.
\end{align*}
Therefore, we have  $[(z_n)_n]=[(w_n)_n]$ in $F(C([0,1], \mathcal{W}\otimes M_{2^{\infty}}
\otimes M_{3^{\infty}}))$ 
since $\mathcal{W}\otimes \mathbb{C}\otimes \mathbb{C}\otimes \mathbb{C} 
\subset \mathcal{W}\otimes C([0,1])\otimes M_{2^{\infty}}\otimes M_{3^{\infty}}$ is a 
nondegenerate inclusion. 
Consequently, (5.1.10) implies that $z:=[(z_n)_n]$ is the desired unitary element. 
\end{proof}

The following proposition is based on \cite[Theorem 5.20]{Sza6} 
(see also \cite[Theorem 5.2]{Sa2}) and \cite[Proposition 3.3]{RW}. 

\begin{pro}\label{pro:prime-drop-embedding}
Let $A$ be a simple separable nuclear monotracial C$^*$-algebra with strict comparison, and let 
$\alpha$ be an automorphism of $A$. 
Then there exists a homomorphism $\sigma$ from $A\otimes Z_{2^{\infty}, 3^{\infty}}$ to 
$A^{\omega}$ such that $\sigma(x\otimes 1_{Z_{2^{\infty}, 3^{\infty}}})=x$ and $\sigma\circ 
(\alpha\otimes \mathrm{id}_{Z_{2^{\infty}, 3^{\infty}}})=\alpha\circ \sigma$ for any $x\in A$.  
\end{pro}
\begin{proof}
By the proof of \cite[Theorem 5.20]{Sza6} (see also \cite[Theorem 5.2]{Sa2}) and 
\cite[Proposition 3.3]{RW}, 
there exists a unital embedding $\varphi$ of $Z_{2^{\infty}, 3^{\infty}}$ into $F(A)^{\alpha}$. 
The Choi--Effros lifting theorem implies that there exists a contractive completely positive map 
$\psi$ from $Z_{2^{\infty}, 3^{\infty}}$ to $A^{\omega}\cap A^{\prime}$ 
 such that $[\psi(y)]=\varphi (y)$ for any 
 $y\in Z_{2^{\infty}, 3^{\infty}}$. 
Define a homomorphism $\sigma$ from  
$A\otimes Z_{2^{\infty}, 3^{\infty}}$ to $A^{\omega}$ by 
$\sigma (x\otimes y):= x\psi (y)$ for any $x\in A$ and 
$y\in Z_{2^{\infty}, 3^{\infty}}$. 
It is easy to see that $\sigma$ is well-defined and has the desired property.  
\end{proof}

The following theorem is one of the main theorems in this section. 

\begin{thm}\label{thm:main-weak-stability}
Let $\alpha$ be an aperiodic automorphism of $\mathcal{W}$. 
If $u$ is a unitary element in $F(\mathcal{W})$ satisfying $\varrho (u)=1$, then there exists 
a unitary element $v$ in $F(\mathcal{W})$ such that $u=v\alpha (v^*)$. In particular, 
every aperiodic automorphism of $\mathcal{W}$ is weakly stable. 
\end{thm}
\begin{proof}
Let $\{F_m\}_{m=1}^{\infty}$ be an increasing sequence of finite subsets of $\mathcal{W}$ such that 
$\overline{\bigcup_{m=1}^{\infty}F_m}=\mathcal{W}$. 
Choose an element $(u_n)_n$ in $\mathcal{W}^{\omega}\cap\mathcal{W}^{\prime}$ such that 
$u=[(u_n)_n]$. By Lemma \ref{lem:key},  there exists a contraction 
$(z_n)_n$ in $(\mathcal{W}\otimes Z_{2^{\infty}, 3^{\infty}})^{\omega}\cap 
(\mathcal{W}\otimes Z_{2^{\infty}, 3^{\infty}})^{\prime}$ such that $[(z_n)_n]$ is a unitary element in 
$F(\mathcal{W}\otimes Z_{2^{\infty}, 3^{\infty}})$ and  
$[(u_n\otimes 1_{Z_{2^{\infty}, 3^{\infty}}})_n]
=[(z_n(\alpha\otimes\mathrm{id}_{Z_{2^{\infty}, 3^{\infty}}})(z_n^*))_n]$. We may assume that 
$z_n$ is a contraction for any $n\in\mathbb{N}$.  
Then we can find a decreasing 
sequence $\{X_m\}_{m=1}^{\infty}$ of elements in $\omega$ such that 
$\bigcap_{m=1}^{\infty} X_m=\emptyset$, 
$$
\| (u_n\otimes 1_{Z_{2^{\infty}, 3^{\infty}}}- z_n(\alpha\otimes\mathrm{id}_{Z_{2^{\infty}, 3^{\infty}}})(z_n^*))
(x\otimes 1_{Z_{2^{\infty}, 3^{\infty}}}) \| <\frac{1}{2^m}
$$
$$
\| (z_n^*z_n-1)(x\otimes 1_{Z_{2^{\infty}, 3^{\infty}}}) \| <\frac{1}{2^m}, 
\quad
\| (z_nz_n^*-1)(x\otimes 1_{Z_{2^{\infty}, 3^{\infty}}}) \| <\frac{1}{2^m}, 
$$
and
$$
\| [z_n, x\otimes 1_{Z_{2^{\infty}, 3^{\infty}}}]\| <\frac{1}{2^m}
$$
for any $n\in X_m$ and $x\in F_{m}$. 
By Proposition \ref{pro:prime-drop-embedding}, there exists a homomorphism $\sigma$ 
from  $\mathcal{W}\otimes Z_{2^{\infty}, 3^{\infty}}$ to 
$\mathcal{W}^{\omega}$ such that $\sigma(x\otimes 1_{Z_{2^{\infty}, 3^{\infty}}})=x$ and $\sigma\circ 
(\alpha\otimes \mathrm{id}_{Z_{2^{\infty}, 3^{\infty}}})=\alpha\circ \sigma$ for any $x\in 
\mathcal{W}$. Then we have
$$
\| (u_n- \sigma (z_n)\alpha(\sigma(z_n)^*))x\| <\frac{1}{2^m}, \quad 
\| (\sigma(z_n)^*\sigma(z_n)-1)x \| <\frac{1}{2^m}, 
$$
$$
\| (\sigma(z_n)\sigma(z_n)^*-1)x \| <\frac{1}{2^m}, 
\quad
\text{and}
\quad
\| [\sigma(z_n), x]\| <\frac{1}{2^m}
$$
for any $n\in X_m$ and $x\in F_{m}$. 
Note that $\sigma (z_n)$ is an element in $\mathcal{W}^{\omega}$. 
For any $n\in X_{m}\setminus X_{m+1}$, we can choose a contraction $v_n$ in $\mathcal{W}$ such 
that 
$$
\| (u_n- v_n\alpha(v_n^*))x\| <\frac{1}{2^m}, \quad 
\| (v_n^*v_n-1)x \| <\frac{1}{2^m}, 
$$
$$
\| (v_nv_n^*-1)x \| <\frac{1}{2^m}, 
\quad
\text{and}
\quad
\| [v_n, x]\| <\frac{1}{2^m}
$$
for any $x\in F_{m}$. Put $v_n:=0$ for any $n\in\mathbb{N}\setminus X_1$. 
Then it can be easily checked that 
$v:=[(v_n)_n]$ is the desired unitary element in $F(\mathcal{W})$. 
\end{proof}

We classify aperiodic automorphisms of $\mathcal{W}$.

\begin{thm}\label{thm:main-classification}
(1) Let $\alpha$ and $\beta$ be aperiodic automorphisms of $\mathcal{W}$. Then 
$\alpha$ and $\beta$ are outer conjugate if and only if 
$N(\tilde{\alpha})=N(\tilde{\beta})$ and $\Lambda(\tilde{\alpha})=\Lambda(\tilde{\beta})$. \ \\
(2) For any $N\in\mathbb{N}$ and $\lambda\in\mathbb{T}$ satisfying $\lambda^{N}=1$, there 
exists an aperiodic automorphism $\gamma$ of $\mathcal{W}$ such that 
$N(\tilde{\gamma})=N$ and $\Lambda (\tilde{\gamma})=\lambda$. 
Furthermore, there exists an aperiodic automorphism $\gamma$ of $\mathcal{W}$ 
such that $N(\tilde{\gamma})=\infty$. 
\end{thm}
\begin{proof}
(1) The ``only if'' part is obvious. We show the ``if'' part. 
By Proposition \ref{pro:char-group}, there exists an aperiodic automorphism $\gamma$ 
of $\mathcal{W}$ such that $N(\tilde{\gamma})=N(\tilde{\alpha})=N(\tilde{\beta})$ and 
$\Lambda (\tilde{\gamma})=\overline{\Lambda (\tilde{\alpha})}=
\overline{\Lambda (\tilde{\beta})}$. 
Proposition \ref{pro:d-absorption}, Theorem \ref{thm:W-absorbing}, and 
Theorem \ref{thm:main-weak-stability} imply that $\alpha$ and $\beta$ are 
$\mathcal{W}$-absorbing. 
By Theorem \ref{thm:absorption-trivial-char}, Theorem \ref{thm:main-weak-stability},  
and an argument similar to that in the proof of Theorem \ref{thm:W-absorbing}, we have 
\begin{align*}
\alpha \sim_{o.c.} \alpha\otimes \mathrm{id}_{\mathcal{W}}\sim_{o.c.} \alpha\otimes 
\gamma\otimes \beta \sim_{o.c.} \mathrm{id}_{\mathcal{W}}\otimes \beta \sim_{o.c.} \beta,
\end{align*}
where the notation $\alpha\sim_{o.c.}\beta$ means that $\alpha$ is outer conjugate to 
$\beta$. Therefore we obtain the conclusion. \ \\
(2) 
By Proposition \ref{pro:model}, there exists an aperiodic automorphism $\gamma_0$ of 
$\mathcal{Q}$ such that $N(\tilde{\gamma}_0)=N$ and $\Lambda (\tilde{\gamma}_0)=\lambda$. 
Then $\gamma:=\gamma_0\otimes \mathrm{id}_{\mathcal{W}}$ of $\mathcal{Q}\otimes \mathcal{W}
\cong \mathcal{W}$ is the desired automorphism. 
The case where $N(\tilde{\gamma})=\infty$ is well-known (see, for example, 
\cite[Section 8]{Na5}). 
\end{proof}

The following corollary is an immediate consequence of Theorem \ref{thm:main-classification}. 

\begin{cor}\label{cor:rigidity}
Let $\alpha$ and $\beta$ be aperiodic automorphisms of $\mathcal{W}$. Then 
$\alpha$ and $\beta$ are outer conjugate if and only if  
$\tilde{\alpha}$ and $\tilde{\beta}$ are outer conjugate.
\end{cor}

Using Theorem \ref{thm:weak-stability-af} instead of Theorem \ref{thm:main-weak-stability}, 
we obtain the following theorem by essentially the same proof as that of 
Theorem \ref{thm:main-classification}. 

\begin{thm}\label{thm:classification-uhf}
Let $B$ be a UHF algebra of infinite type, and let $\alpha$ and $\beta$ be aperiodic automorphisms 
of $B$. Then $\alpha\otimes \mathrm{id}_B$ and $\beta\otimes 
\mathrm{id}_B$ are outer conjugate if and only if 
$N(\tilde{\alpha})=N(\tilde{\beta})$ and $\Lambda(\tilde{\alpha})=\Lambda(\tilde{\beta})$.
\end{thm}

\begin{rem}
(1) By Theorem \ref{thm:main-classification}, we can also completely classify aperiodic automorphisms 
of $\mathcal{W}\otimes\mathbb{K}$, where $\mathbb{K}$ 
denotes the C$^*$-algebra of compact operators on an infinite-dimensional separable Hilbert space. 
Indeed, trace scaling automorphisms of $\mathcal{W}\otimes\mathbb{K}$ are completely classified in 
\cite[Theorem 7.3]{Na2}. 
Let $\alpha$ be a trace preserving aperiodic automorphism of $\mathcal{W}\otimes\mathbb{K}$. 
Since $\alpha$ is trace preserving,  we have 
$d_{\tau\otimes \mathrm{Tr}}(\alpha(1_{M(\mathcal{W})}\otimes e_{11}))
=d_{\tau\otimes \mathrm{Tr}}(1_{M(\mathcal{W})}\otimes e_{11})$, where $e_{11}$ is a rank 
one projection in $\mathbb{K}$, and $\mathrm{Tr}$ is the usual trace on $\mathbb{K}$. 
By \cite[Proposition 2.9]{Na0}, $\alpha(1_{M(\mathcal{W})}\otimes e_{11})$ is Murray--von Neumann 
equivalent to $1_{M(\mathcal{W})}\otimes e_{11}$ in $M(\mathcal{W}\otimes \mathbb{K})$. 
Since $1_{M(\mathcal{W}\otimes \mathbb{K})}-1_{M(\mathcal{W})}\otimes e_{11}$ is 
Murray--von Neumann equivalent to $1_{M(\mathcal{W}\otimes \mathbb{K})}$, 
$1_{M(\mathcal{W}\otimes \mathbb{K})}-\alpha(1_{M(\mathcal{W})}\otimes e_{11})$ 
is Murray--von Neumann equivalent to 
$1_{M(\mathcal{W}\otimes \mathbb{K})}-1_{M(\mathcal{W})}\otimes e_{11}$. 
Hence there exists a unitary element 
$u$ in $M(\mathcal{W}\otimes\mathbb{K})$ such that 
$u\alpha(1_{M(\mathcal{W})}\otimes e_{11})u^*=1_{M(\mathcal{W})}\otimes e_{11}$. Define an 
automorphism $\gamma$ of $\mathcal{W}$ by
$$
\gamma (x):= \mathrm{Ad}(u)\circ \alpha (x\otimes e_{11})
$$
for any $x\in \mathcal{W}\cong (1_{M(\mathcal{W})}\otimes e_{11})
(\mathcal{W}\otimes \mathbb{K})
(1_{M(\mathcal{W})}\otimes e_{11})$. Then $\alpha$ is outer conjugate to 
$\gamma \otimes \mathrm{id}_{\mathbb{K}}$. Hence the classification of 
trace preserving aperiodic automorphisms of $\mathcal{W}\otimes\mathbb{K}$ can be reduced to 
the classification of aperiodic automorphisms of $\mathcal{W}$. 
\ \\
(2) If we can show that every aperiodic automorphism of $\mathcal{Z}$ (resp. a UHF algebra 
$B$ of infinite type) is weakly stable, then we can classify aperiodic automorphisms 
of $\mathcal{Z}$ (resp. $B$) up to outer conjugacy 
by essentially the same argument as in the proof of Theorem \ref{thm:main-classification}(1). 
Note that the range of the invariant is a separate problem. 
\end{rem}

\section*{Acknowledgments}
The author would like to thank Yasuhiko Sato and Toshihiko Masuda for helpful discussions  
at a seminar held at Kyushu University. In particular, the definition of weak stability was inspired by 
these discussions. 

The author used Claude Opus 5.5, an AI model developed by Anthropic,
in the preparation of this paper. It was used to check the English and the structure of sentences, 
to find typographical errors, and to check details of some proofs. 
In addition, the idea of Lemma \ref{lem:Opus} and its proof was suggested by Claude,
and the author verified it and rewrote the proof in the present form. 
Claude was also used to simplify the proofs of Proposition \ref{pro:intertwining} and 
Theorem \ref{thm:path}. 
The author has independently verified all the mathematical arguments
in this paper and takes
full responsibility for the content of this paper.

\end{document}